\documentclass{amsart}
\usepackage{graphicx}
\usepackage{amsmath}
\usepackage{subfigure}
\usepackage{amssymb} 
\usepackage[usenames,dvipsnames]{color}
\usepackage{pinlabel}
\newtheorem{thm}{Theorem}[section]

\newtheorem{lemma}[thm]{Lemma}

\newcommand{\bdd}{\mbox{$\partial$}}

\def\R{{\mathbb R}}

\begin{document}  

\title{Proposed Property 2R counterexamples classified}   

\author{Martin Scharlemann}
\address{\hskip-\parindent
        Martin Scharlemann\\
        Mathematics Department\\
        University of California\\
        Santa Barbara, CA USA}
\email{mgscharl@math.ucsb.edu}

\thanks{Research partially supported by National Science Foundation grants.}

\date{\today}

\begin{abstract} In 1985 Akbulut and Kirby \cite{AK} analyzed a homotopy $4$-sphere $\Sigma$  that was first discovered by Cappell and Shaneson \cite{CS}, depicting it as a potential counterexample to three important conjectures, all of which remain unresolved.  In 1991, Gompf's further analysis \cite{Go} showed that $\Sigma$ was one of an infinite collection of examples, all of which were (sadly) the standard $S^4$, but with an unusual handle structure. 

In \cite{GST} it was shown that the construction gives rise to a family $L_n$ of $2$-component links, each of which remains a potential counterexample to the generalized Property R Conjecture.  In each $L_n$ one component is the simple square knot $Q$ and it was argued that the other component, after handle-slides, could in theory be placed very symmetrically.  How to accomplish this was unknown, and that question is resolved here, in part by finding a symmetric construction of the $L_n$.  In view of the continuing interest and potential importance of the Cappell-Shaneson-Akbulut-Kirby-Gompf examples (e. g. the original $\Sigma$ is known to embed very efficiently in $S^4$ and so provides unique insight into proposed approaches to the Schoenflies Conjecture) digressions into various aspects of this enhanced view are also included.   \end{abstract}

\maketitle

In 1987 David Gabai, demonstrating the power of his newly-developed sutured manifold theory, settled the long-sought Property R Conjecture \cite{Ga}:

\begin{thm}[Property R] \label{thm:PropR} If $0$-framed surgery on a knot $K \subset S^3$ yields $S^1 \times S^2$ then $K$ is the unknot.
\end{thm}

The conjecture had been motivated in part by a simple question:  Could a non-standard homotopy $4$-sphere be built using just a single $1$-handle and a single $2$-handle?  Gabai's theorem established that it could not, but it then begs two more general questions: 
\begin{itemize}
\item Is a homotopy $4$-sphere built without $3$-handles necessarily standard?  This is a generalization because such a manifold has (a single $0$-handle and) the same number of $2$-handles as $1$-handles. 
\item If it is always standard, can one show this by simple handle-slides, without introducing additional pairs of canceling handles?
\end{itemize}

In the second, stronger form, this is the Generalized Property R Conjecture.  usually stated \cite[Problem~1.82]{Ki2}, as a conjecture about links: If surgery on an $n$-component link $L$ yields the connected sum $\#_nS^1\times S^2$ then $L$ becomes the unlink after suitable handle slides.   \cite{GST} focused on the the hunt for a counter-example in the simplest possible case. In an argument that can be traced back 30 years to \cite{AK}, it  is argued that a counter-example is likely to arise even among those $2$-component links in which one component is simply the square knot $Q$.  Here is a brief review:
\bigskip

Let $\mathcal{L}$ denote the set of  of all two-component links that contain the square knot $Q$ and on which some surgery yields $(S^1 \times S^2) \# (S^1 \times S^2)$.   A combination of sutured manifold theory and Heegaard theory provides  a natural classification of $\mathcal{L}$, up to handle slides over $Q$.  In brief: for any link $Q \cup K \subset \mathcal{L}$,  handle-slides of $K$ over $Q$ will eventually transform $K$ to a knot that lies in the standard genus $2$ Seifert surface $F$ of $Q$, and furthermore its position in $F$ is very constrained: viewing $F$ as the $2$-fold branched cover of a $4$-punctured sphere $P$, $K$ must be a homeomorphic lift of an embedded circle in $P$.  There is a natural way of parameterizing  embedded circles in $P$ by the rationals; with one such parameterization, a circle in $P$ lifts homeomorphically to $F$ if and only if the corresponding rational has odd denominator. 

A frustrating aspect of this classification is that each handle-slide of $K$ over $Q$ requires the choice of a band over which to slide, and the classification does not give a prescription for finding the relevant bands.  So there is no direct way to see how a given link in $\mathcal{L}$ fits into the classification.  In particular, it is also argued in \cite{GST} that  there is an infinite family $L_n \subset \mathcal{L}$ of links that probably do not satisfy the Property 2R Conjecture, but we were unable to resolve how the family $L_n$ fits into the classification of $\mathcal{L}$.  This left a puzzling gap  (\cite[Question One]{GST}) between the $4$-dimensional Kirby-calculus arguments which gave rise to interest in $L_n$, and the $3$-dimensional sutured manifold arguments which were used to classify $\mathcal{L}$.  Here we resolve that question by showing that each $L_n$ corresponds to the slope $\frac{n}{2n+1} \in \mathbb{Q}$. 
\bigskip

The family $L_n = Q \cup V_n$ (called $L_{n, 1}$ in \cite{GST}) has these central features:  
 \begin{itemize}
 \item If $L_n$ satisfies the generalized Property R Conjecture then a certain well-known group presentation could be trivialized by Andrews-Curtis moves, an outcome that seems very unlikely when $n \geq 3$.  
 \item  On the other hand, $L_0$ can be handle-slid to become the unlink. (This is demonstrated in  \cite[Figures 12, 13, 5]{GST}.)
 \item  Let $M$ denote the $3$-manifold obtained from $S^3$ by $0$-framed surgery on $Q$.  There is a torus $T \subset M$ and simple closed curve $\alpha \subset T$ so that each $V_n$ intersects $T$ twice and Dehn twisting $V_n$ in $M$ at $T$ along the slope given by $\alpha$ converts $V_n$ to $V_{n+1}$. 
 \item When viewed in $S^3$, the framing of $\alpha$ given by its annular neighborhood in $T$ is $\pm 1$.  
 \end{itemize}  
 
It is further shown that, given these last two features, $L_{n+1}$ can be obtained from $L_n$ by handle-slides and the introduction and use of a single canceling Hopf pair.  It follows inductively that, {\em with the introduction of a canceling Hopf pair}, each $L_n$ can be handle-slid to the unlink.    In particular, $0$-framed surgery on each $L_n$ gives $(S^1 \times S^2) \# (S^1 \times S^2)$, and the corresponding homotopy $4$-sphere is standard, because its handle structure can be trivialized with the introduction of a cancelling $2$- and $3$-handle pair.  
\bigskip

Here is an outline of the paper:  Section \ref{sect:explicit} is aimed at filling a pictorial gap in \cite{GST}, namely the transition from \cite[Figure 12d]{GST} (itself derived from \cite[Figure 1]{Go}) to the full-blown explicit example \cite[Figure 1]{GST}.   Section \ref{sect:push} begins the process of moving (perhaps by slides) the component $V_n$ of $L_n$ onto the Seifert surface $F$ of the square knot $Q$. The modest goal here is to provide a highly symmetric picture, in which $F$ is clearly displayed, with $V_n$ nearby.   In order to complete the move of $V_n$ onto $F$, Section \ref{sect:hex} changes the viewpoint: instead of viewing $Q$ and $V_n$ in $S^3$, we see how $V_n$ lies in the fibered manifold $M$ obtained by doing $0$-framed surgery on $Q$, and move it onto a fiber.  This accomplished, Section \ref{sect:classification} describes how the resulting link fits into the classification of $\mathcal{L}$ developed in  \cite{GST}.   

The torus $T$ described above plays a crucial role in the construction of the $V_n$, but $T$ itself is difficult to see.  The rest of the paper aims for a clearer picture of $T$.   Section \ref{sect:alternate} provides an alternate view of of the construction of $V_n$, a view in which $T$ is more easily tracked.  This leads in Section \ref{sect:torus} to an explicit description of $T$ in the fibered manifold $M$.  This, in turn, leads in Section \ref{sect:present} to a pleasant description of how the Andrews-Curtis problematic presentation of the trivial group naturally arises from the construction of $L_n$.  In particular the relation $aba = bab$ in that presentation derives from the natural presentation of $\pi_1(D^4 - D_Q)$, where $D_Q$ is the ribbon disk that $Q$ bounds in $D^4$.  Finally, in Section \ref{sect:TinS3}, a fuller description of how $T$ appears in $S^3$ is given, one that includes an explicit picture of the slope in $T$ along which Dehn twisting changes $V_n$ to $V_{n+1}$.  
 \bigskip
 
 Many of the arguments in the paper are pictorial; some are likely to be indecipherable without viewing the full-color rendering in the on-line .pdf file.  

\section{From surgery diagram to explicit link}\label{sect:explicit}

We begin by explaining with pictures the transition in \cite{GST} from a fairly simple surgery diagram of a probable counterexample to Property 2R to an explicit picture of it as a complicated link in $S^3$.   Begin with  the surgery diagram Figure \cite[Figure 12d]{GST} and draw it symmetrically, a transition illustrated in Figure \ref{fig:prehexstep1}.  

Next consider the effect of blowing down the red circles in the surgery description labelled $[\pm 1]$.  Focus on the right circle and recall the easy fact (see \cite[Figure 5.18]{GS}) that blowing down the circle with label $[-1]$ is equivalent to taking the disk $D$ that it bounds and giving everything that runs through it a $+1$ (sic) twist.  In this case, part of what runs through $D$ are $n$ segments of a $0$-framed component coming from the $+n$ twist-box on the right.  Figure \ref{fig:prehexstep2} is meant to illustrate what happens, very specifically in the top row for $n = 2$.  Before twisting along $D$, move the $n=2$ punctures in $D$ so that as one moves clockwise around the disk, the punctures become more central.  The track of the $2$ strands going through $D$,  after the $+1$ twist around $D$, is shown in red.  It is then illustrated in blue how the result can be thought of as a $+1$ twist box placed around an $n$-stranded band that follows the original red circle that was blown down.  The result for general $n$ is shown next, and then this is applied to the given surgery diagram to give the link illustrated at the bottom of Figure \ref{fig:prehexstep2}.  The gray annuli are meant to represent $n$ parallel strands, much like the wide annuli that appear in \cite[Figure 1]{GST}.  We now discuss the transition to that figure.

\begin{figure}
     \labellist
\small\hair 2pt
\pinlabel  $n$ at 175 15
\pinlabel  $n$ at 555 95
\pinlabel  $-n$ at 15 75
\pinlabel  $-n$ at 325 65
\pinlabel  $[1]$ at 65 110
\pinlabel  $[1]$ at 370 120
\pinlabel  $[-1]$ at 210 95
\pinlabel  $[-1]$ at 500 120
\pinlabel  $0$ at 195 155
\pinlabel  $0$ at 205 130
\pinlabel  $0$ at 430 170
\pinlabel  $0$ at 430 140
\endlabellist
 \centering
    \includegraphics[scale=0.6]{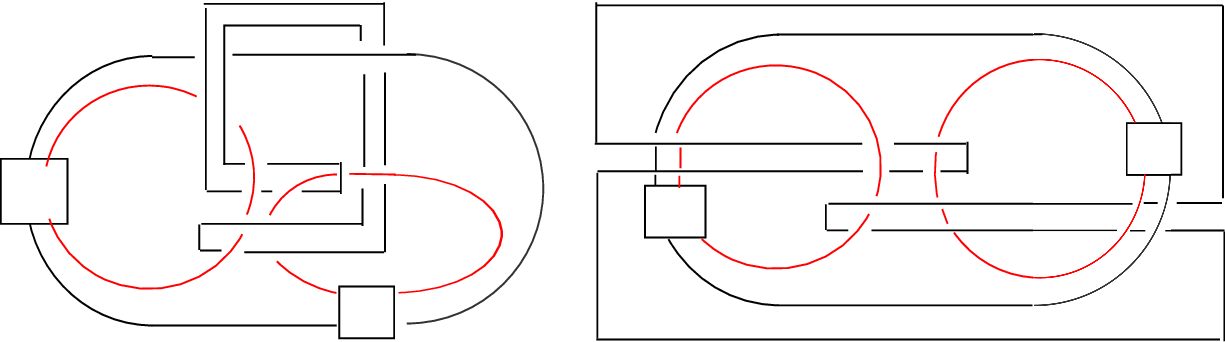}
    \caption{} \label{fig:prehexstep1}
    \end{figure}

\begin{figure}
     \labellist
\small\hair 2pt
\pinlabel  $2$ at 18 343
\pinlabel  $n$ at 23 215
\pinlabel  $n$ at 275 265
\pinlabel  $[-1]$ at 5 240
\pinlabel  $[-1]$ at 0 380
\pinlabel  $[-1]$ at 70 380
\pinlabel  $+1$ at 283 332
\pinlabel  $+1$ at 278 220
\pinlabel  $+1$ at 260 90
\pinlabel  $-1$ at 35 65
\endlabellist
 \centering
    \includegraphics[scale=0.8]{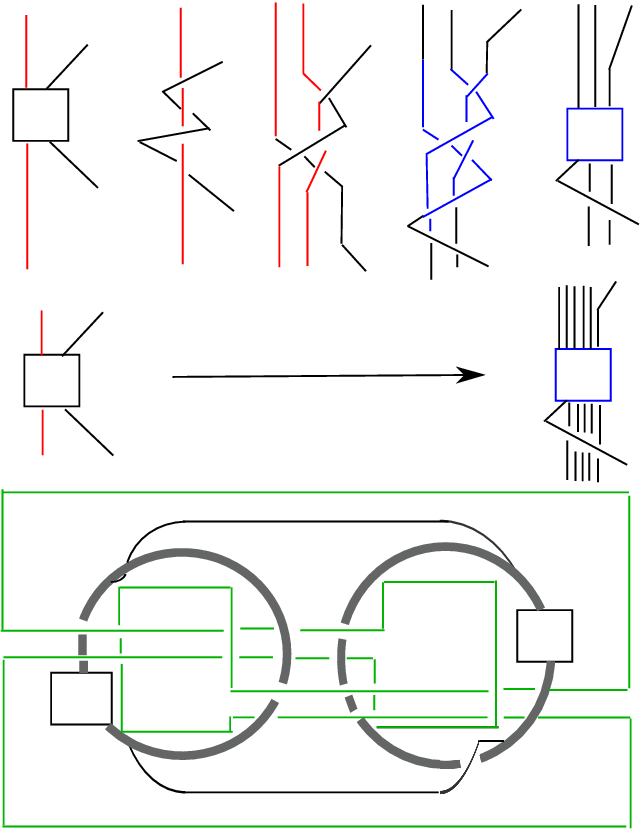}
    \caption{} \label{fig:prehexstep2}
    \end{figure}

In Figure \ref{fig:prehexstep5a}, the gray annuli are pushed off the plane of the green knot (now visibly the square knot), then the twist boxes are moved clockwise past one end of the two (black) connecting arcs, switching a crossing between the arc and the annulus to which it is connected.  Finally, the two gray annuli are isotoped to push the parts containing the twist boxes to the outside of the figure, beginning to imitate their positioning in \cite[Figure 1]{GST}.

\begin{figure}
 \centering
    \includegraphics[scale=0.8]{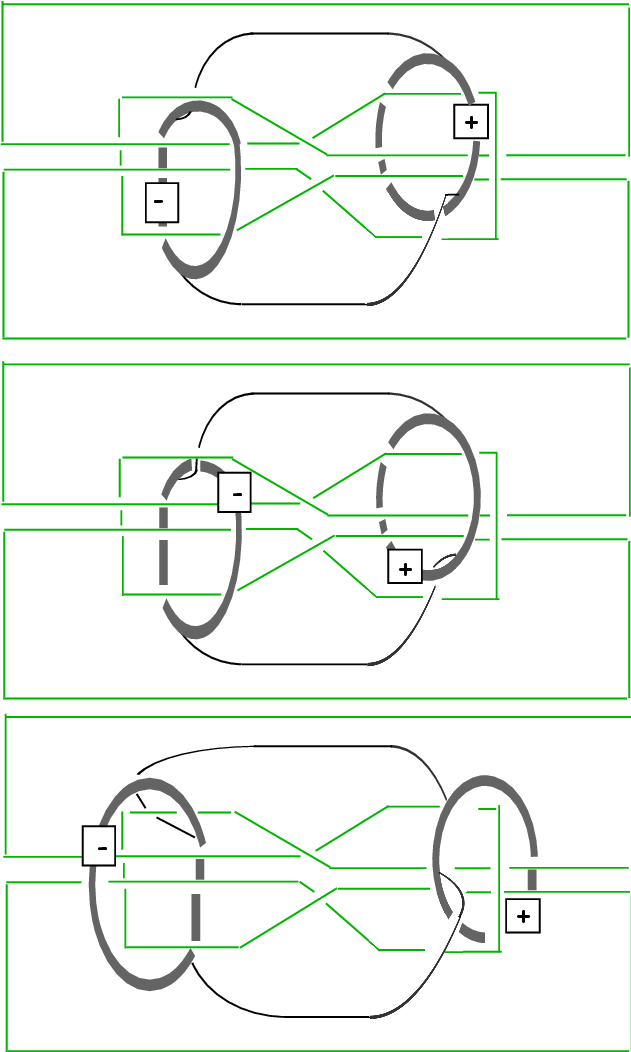}
    \caption{} \label{fig:prehexstep5a}
    \end{figure}

Continue with the positioning, moving clockwise from the upper left in Figure \ref{fig:prehexstep5b}.  The first step there is to move the square knot into the correct position and then expand the two gray annuli.  Next the ends of the connecting arcs on each annuli are moved to be adjacent (represented by the small blue squares in the figure).  This has the psychological advantage that the link component $V_n$ can be thought of as starting with a single circle (the end-point union of the two black arcs in the figure) and then doing a $\pm n$ Dehn twist to that circle along the cores of the two gray annuli.  The next move is a surprise: push the blue square in the right hand gray annulus clockwise roughly three-quarters around the annulus, pushing the twist-box ahead of it.  When this is done (the bottom left hand rendering) the picture of the link has become essentially identical to that in  \cite[Figure 1]{GST}.

\begin{figure}
 \centering
    \includegraphics[scale=0.6]{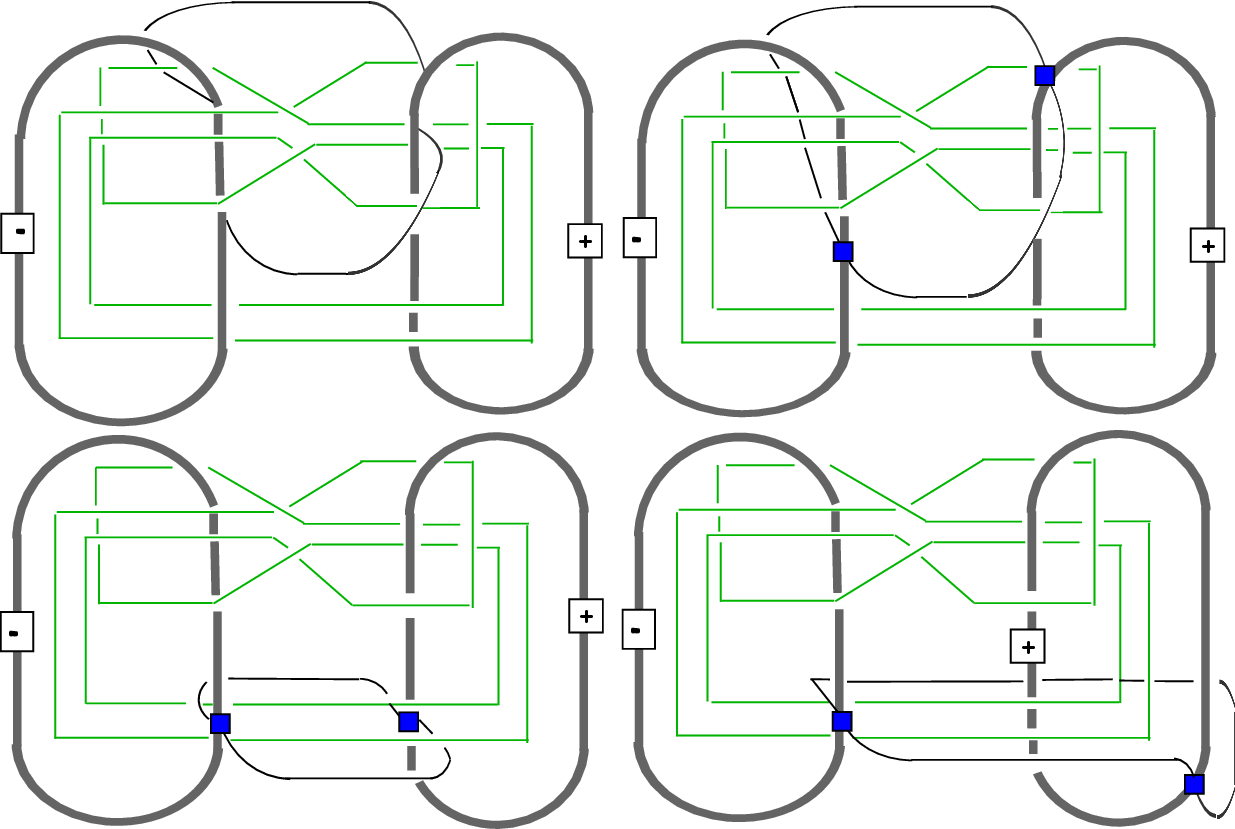}
    \caption{} \label{fig:prehexstep5b}
    \end{figure}

\section{Pushing $V_n$ onto a Seifert surface}\label{sect:push}

It follows easily from \cite[Corollary 4.2]{ST} that each $V_n$, perhaps after some handle-slides over $Q$, can be placed onto a standard Seifert surface of $Q$, so that the framing of $V_n$ given by the Seifert surface coincides with $0$-framing in $S^3$.  (For details, see the proof of \cite[Theorem 3.3]{GST}.)  The proof of \cite[Corollary 4.2]{ST} requires Gabai's deep theory of sutured manifolds.  It is non-constructive and in particular the proof provides no description of how to find the handle-slides of $V_n$ over $Q$ that are needed to place $V_n$ onto the Seifert surface.  
%So it was not surprising that we could not find the appropriate handle-slides in \cite{GST}, even though we could prove that they exist.  

In this section we begin the search for the required handle-slides by trying to push the whole apparatus $A$ that defines $V_n$ onto a Seifert surface $F$ for $Q$. (Here $A$ is the union of the two gray annuli in Figure \ref{fig:prehexstep2} with the two arcs that connect them.)  It's immediately clear that $A$ can't be completely moved onto $F$, since each of the gray annuli in $A$ has non-trivial linking number with $Q$.  Here we focus on getting $A$ as closely aligned with $F$ as seems possible, given this difficulty.

\begin{figure}
\centering
\subfigure[]%caption in brackets
{
  \includegraphics[scale=0.5]{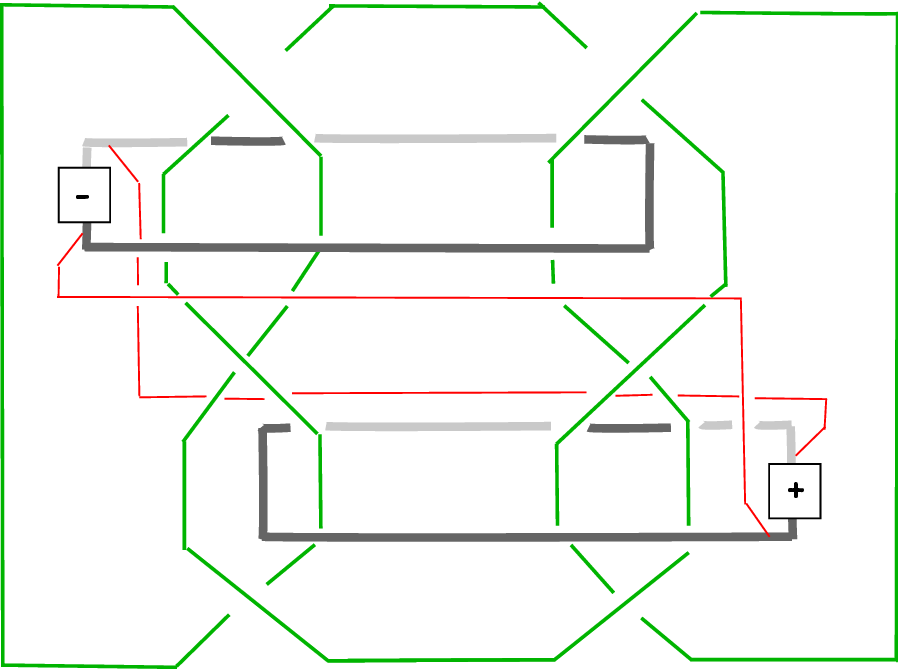} \label{fig:prehexsquare0}
}
\subfigure[]%caption in brackets
{
 \includegraphics[scale=0.5]{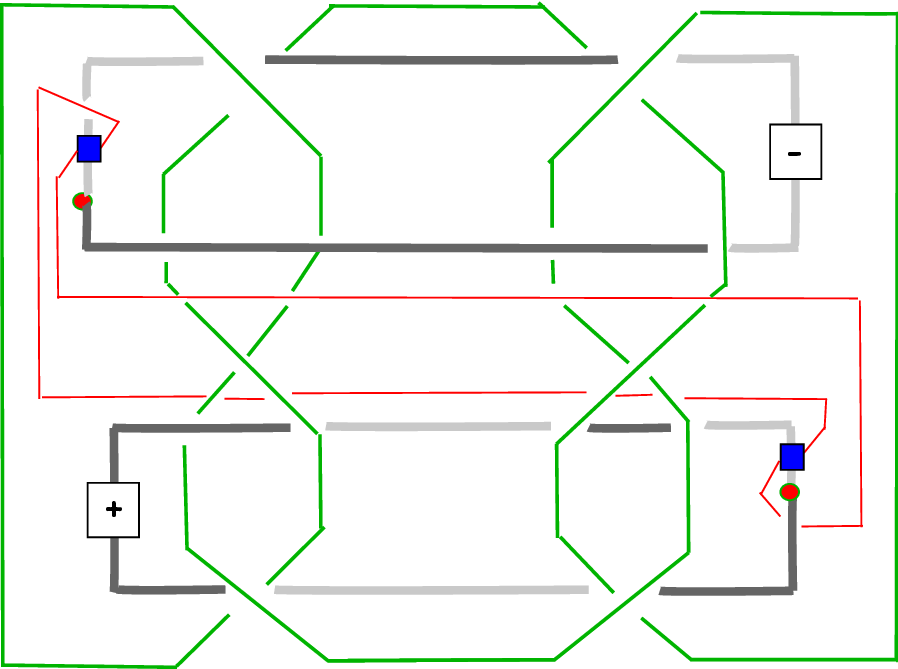} \label{fig:prehexsquare2}
}
\caption{}
\label{fig:prehexsquare02}
\end{figure}

Figure \ref{fig:prehexsquare02} repeats the first two steps in Figure \ref{fig:prehexstep5a}, with these minor modifications:
\begin{itemize}
\item  The square knot $Q$ is laid out in the classic fashion that emphasizes its Seifert surface: three rectangles in the plane of the page, with the middle rectangle joined to each of the side rectangles via three twisted bands. 
\item The Seifert surface  $F$ is not shown; instead, the annuli have been shaded to show where they lie in front of or behind $F$.  
\item In the lower figure, not only has each twist box been passed through an end  of one of the arcs (as was done in Figure \ref{fig:prehexstep5a}), but also the annuli have been stretched horizontally and the top subarc of the top annulus passed over the top subarc of $Q$ (and symmetrically at the bottom) so that much of each annulus lies right next to $F$.  
\item Since each gray annulus links $Q$ once, there is an intersection point (shown with a red dot in the lower figure) of each with $F$.  
\item The arcs connecting the annuli appear in red.
\end{itemize}

\begin{figure}
\centering
\subfigure[]%caption in brackets
{
  \includegraphics[scale=0.5]{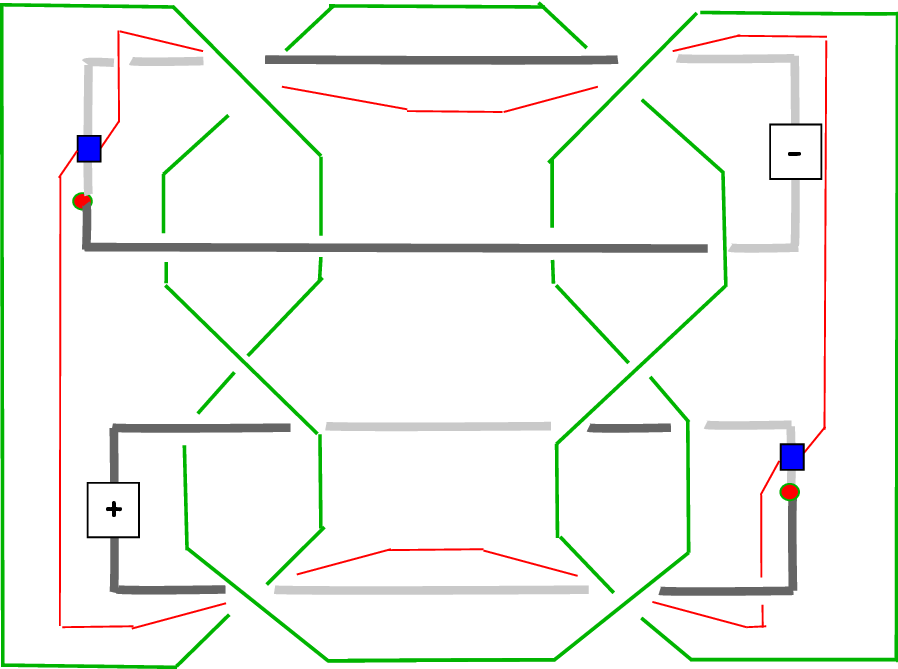} \label{fig:prehexsquare4}
}
\subfigure[]%caption in brackets
{
 \includegraphics[scale=0.5]{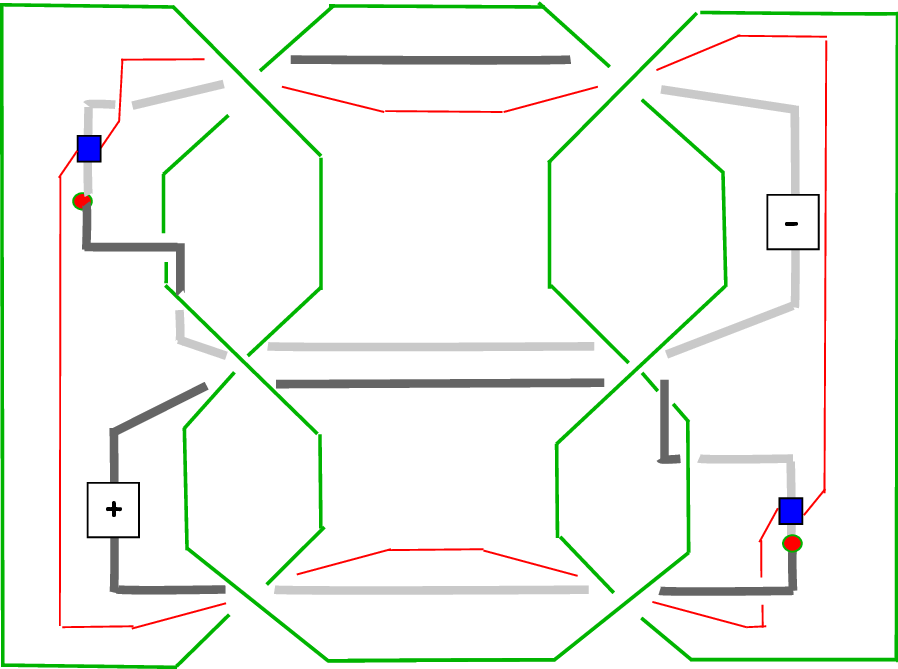} \label{fig:prehexsquare5}
}
\caption{}
\label{fig:prehexsquare45}
\end{figure}

Further progress is shown in Figure \ref{fig:prehexsquare45}: first the red arcs are moved mostly onto $F$ by passing one of them over the top (and the other under the bottom) of the center of $Q$.  (In fact, if the red arcs were not attached to the annuli at the blue 
squares, but set free to form their own circle, that circle could be moved entirely onto $F$ away from the annuli.)  Then the subarcs of the annuli at the center of the figure are moved to conform to $F$ as much as possible, at one point inevitably passing from the top of 
$F$, around $Q$ and then under $F$.  Each annulus conforms even more closely to $F$ if half of its full twist is absorbed into the point at which  the annulus is passed around $Q$, moving from the top of $F$ to the bottom.  The final result is shown in Figure \ref{fig:prehexsquare6}.  

Sadly, central features of Figure \ref{fig:prehexsquare6} do not offer much hope of further pushing $V_n$ onto $F$.  Most noticeably, the gray annuli which carry so much of $V_n$ have non-trivial linking number with $Q$, so there is no way that they can be moved intact onto the Seifert surface $F$.  Of course we are allowed to change $V_n$ by sliding it over $Q$, but this is very difficult to picture in $S^3$.  Further progress in pushing $V_n$ onto $F$ requires a good understanding of the fiber structure of the complement of $Q$ and the parallel fiber structure of the closed $3$-manifold $M$ obtained by $0$-framed surgery on $Q$.  That is the subject of the next section.

\begin{figure}
\centering
\labellist
\small\hair 2pt
\pinlabel  $-\frac12$ at 42 250
\pinlabel  $-\frac12$ at 97 200
\pinlabel  $\frac12$ at 315 108
\pinlabel  $\frac12$ at 385 53
\endlabellist
 \includegraphics[scale=0.7]{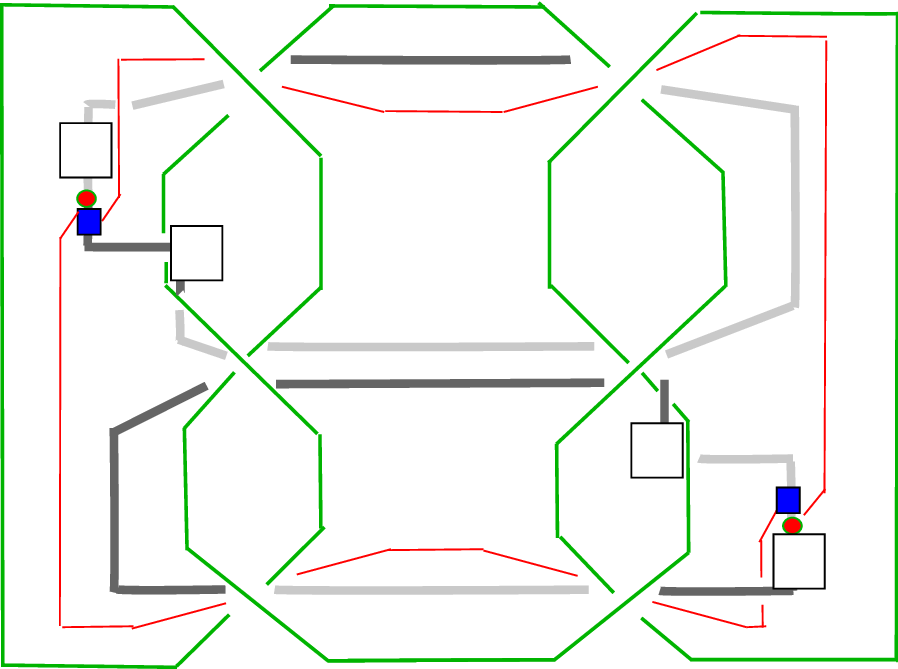}
\caption{} \label{fig:prehexsquare6}
\end{figure}

\section{The fiber structures of $S^3 - Q$ and $M$} \label{sect:hex}

For understanding the structure of the fibering of $S^3 - Q$, a good starting point is the fibering of the simple trefoil knot, for $Q$ is the sum of two of these. (In discussing knot complements, we'll suppress the distinction between a knot and its regular neighborhood in $S^3$.  Thus $S^3 - Q$ will be shorthand for the compact manifold with torus boundary that is the complement in $S^3$ of an open regular neighborhood of $Q$.)   

A nice account of this fibering is given in \cite[Section 3]{Z} beginning with Zeeman's modest: ``I personally found it hard to visualise 
how the complement of a knot could be fibered so beautifully, until I heard a talk by John Stallings on Neuwirth knots.''  

The classic description of the trefoil complement's fibering begins with viewing its punctured torus Seifert surface as two disks connected by three twisted bands.  (A possibly confusing feature:  the sign of the twist on the band is actually the opposite of the sign of the knot crossing, so, for example, the trefoil summand pictured on the left of $Q$ is a left-handed trefoil knot $Tr_-$, but the twists on the three bands in its Seifert surface $S_-$ are right-handed.)  The monodromy of $Tr_-$ cycles the three bands and exchanges the two disks, and so ends up being of order six.  At each iteration of the monodromy the knot itself is rotated a sixth of the way along itself.  A dual view of the monodromy will be more useful here:  pick one vertex in the center of each of the two disks in $S_-$ and connect these two vertices by three edges, each running through one of the twisted bands in $S_-$. (This is roughly shown Figure \ref{fig:posthexsquare1}.) If $S_-$ is cut open along this $\theta$ graph, the result can be viewed as a planar hexagon with a disk removed. The circle boundary of the removed disk corresponds to $Tr_-$, and $S_-$ itself can be recovered by identifying opposite sides of the hexagon, as in Figure \ref{fig:trefoilmon}.   The period six monodromy is then just a $\frac{\pi}3$ rotation of the punctured hexagon and we can view $S^3 - Tr_-$ as the mapping torus of $S_-$ under this monodromy.  If we were interested in the closed manifold obtained by $0$-framed surgery on $Tr_-$, the picture would be the same, but with the disk filled in.  

The square knot $Q$ is the connected sum of two trefoil knots $Tr_{\pm}$, one right-handed and the other left-handed. One way to see the monodromy of the Seifert surface $S_+$ of the other trefoil $Tr_+$ is to use Figure \ref{fig:trefoilmon}, but, in order to obtain the opposite orientation, first reflect the figure across the inner (circle) boundary component, so the 
hexagonal boundary lies on the inside and the knot boundary $Tr_+$ on the outside, as in Figure \ref{fig:trefoilminus}.   The complement of $Tr_+$ is then the mapping torus of $S_+$ under the monodromy shown in Figure \ref{fig:trefoilminus}.  

Since $Q$ is the connected sum of $Tr_{\pm}$ it is easy to see that $S^3 - Q$ can be obtained by gluing the manifolds $S^3 - Tr_+$ and $S^3 - Tr_-$ along a meridional annulus in the boundary of each.  A meridional annulus in the mapping torus picture of the knot complements is the mapping torus of a subinterval of $\bdd S_\pm$, so the result of gluing together the two knot complements along a meridional annulus of each is a fibering of $Q = Tr_+ \# Tr_-$ in which the fiber is the punctured genus $2$ surface $F = S_+  \natural  S_-$.  The monodromy acts on $F$ as shown in Figure \ref{fig:Qmon}.  The monodromy is a bit more complicated than it first appears (on the left in Figure \ref{fig:Qmon}): in order to make the monodromy preserve the boundary circle (corresponding to the knot $Q$) the $\frac{\pi}3$ rotation must be undone near the central circle, as shown on the right of Figure \ref{fig:Qmon}. This complication disappears in a description of the manifold $M$ however, because the boundary circle of $F$ is filled in with a disk, so $M$ is simply the mapping cylinder of the $\frac{\pi}3$ rotation shown in Figure \ref{fig:Mmon}.  The figure shows how the monodromy acts on the fiber $F_\cup = F \cup_{\bdd} D^2$ of $M$, obtained by filling in the boundary of $F$ with a disk.  

\begin{figure}
\centering
\subfigure[Structure of $S^3 - Tr_-$]%caption in brackets
{
   \includegraphics[scale=0.3]{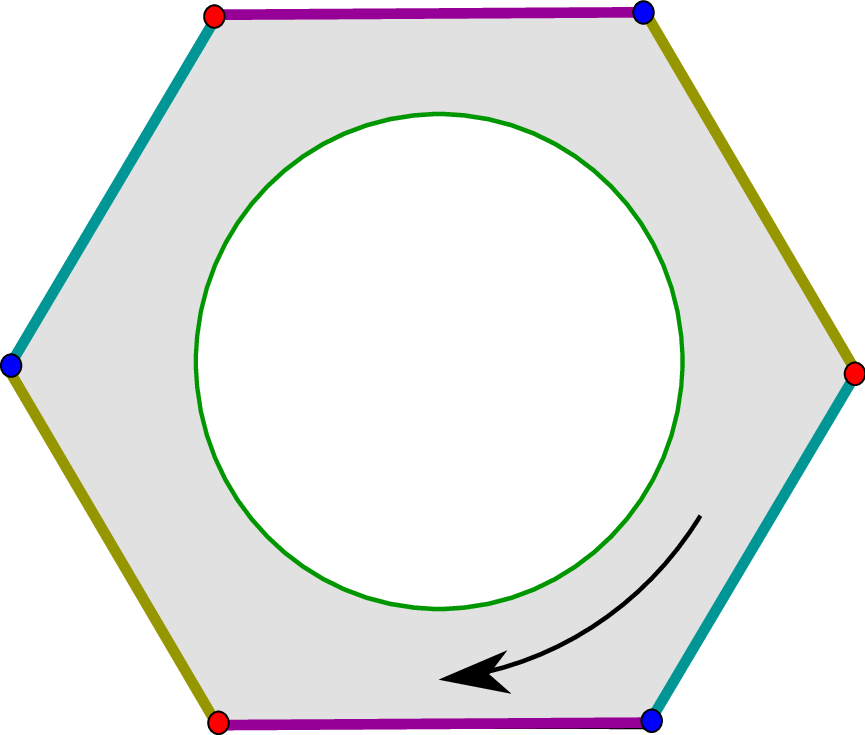}
   \label{fig:trefoilmon}
}
\subfigure[Of $S^3 - Tr_+$]%caption in brackets
{
  \includegraphics[scale=0.35]{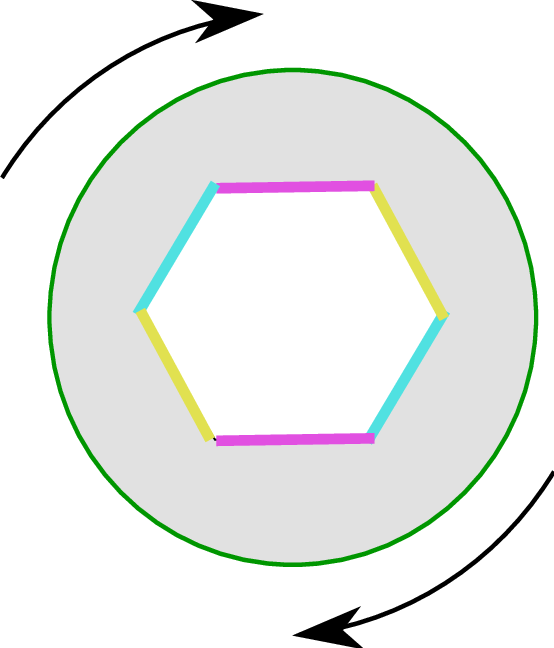}
   \label{fig:trefoilminus}
}
\caption{}
\label{fig:TreMmon}
\end{figure}

\begin{figure}
 \centering
    \includegraphics[scale=0.5]{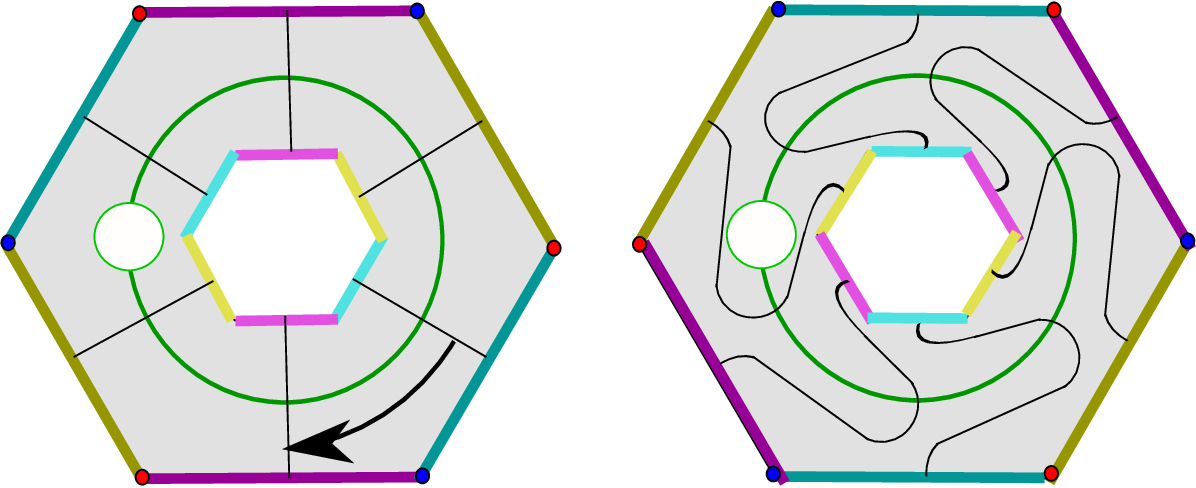}
    \caption{} \label{fig:Qmon}
    \end{figure}

\begin{figure}
 \centering
    \includegraphics[scale=0.4]{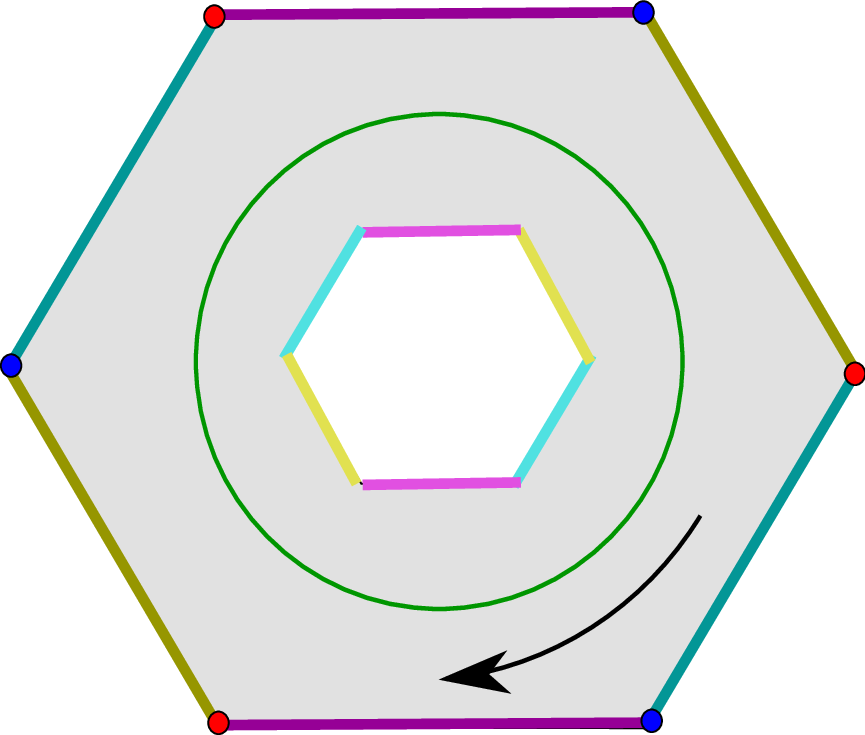}
    \caption{Structure of $M$} \label{fig:Mmon}
    \end{figure}

To understand how the monodromy acts on $F$ when viewed in $S^3$ consider how the monodromy of the trefoil knot acts on, say, the left half $F_\ell$ of Figure \ref{fig:prehexsquare5} (corresponding to $S_-$) and double this to get the action on all of $F$.  
This process is described in three steps in Figure \ref{fig:posthexsquare12}.  On the left are shown how the three arcs corresponding to the pairwise identified sides of the hexagon in Figure \ref{fig:trefoilmon} appear in $F_\ell$. 
The arcs connect the red vertex to the blue vertex and are oriented so that the monodromy takes the top arc to the middle arc and the middle arc to the bottom arc. The monodromy also takes the bottom arc to the top arc, but reverses the orientation, reflecting the fact that the monodromy is of order six.  Red, blue and green properly embedded transverse arcs in $F_\ell$ are added to the figure, and in Figure  \ref{fig:posthexsquare1b} these are slid along the trefoil knot until their ends lie on the vertical arc along which $F_\ell$ is doubled to recover $F$ (Figure \ref{fig:posthexsquare2}).  When the red, blue and green arcs are doubled they become three circles in $F$ which are permuted by the monodromy on $F_\cup$.  Figure \ref{fig:compare} shows the three circles as they appear in $F \subset F_\cup$, the fiber of $M$, and how they appear in $F$, the Seifert surface of $Q$ in $S^3$.

\begin{figure}
\centering
\subfigure[Figure \ref{fig:trefoilmon} in $S^3$]%caption in brackets
{
   \includegraphics[scale=0.4]{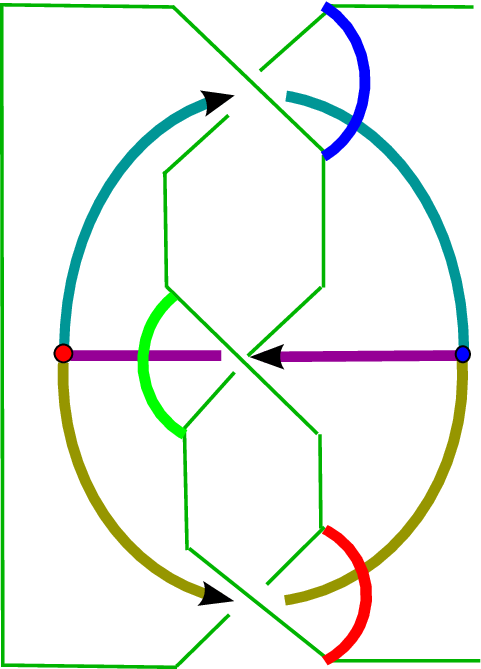}
   \label{fig:posthexsquare1}
}
\subfigure[Sliding]%caption in brackets
{
  \includegraphics[scale=0.4]{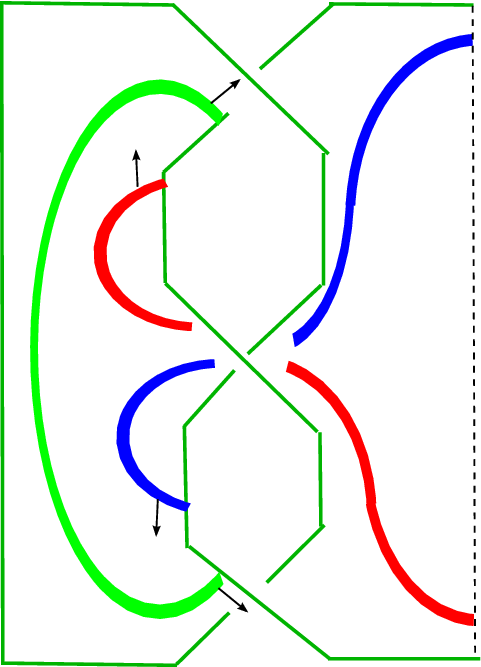}
   \label{fig:posthexsquare1b}
}
\subfigure[The transverse arcs]%caption in brackets
{
  \includegraphics[scale=0.4]{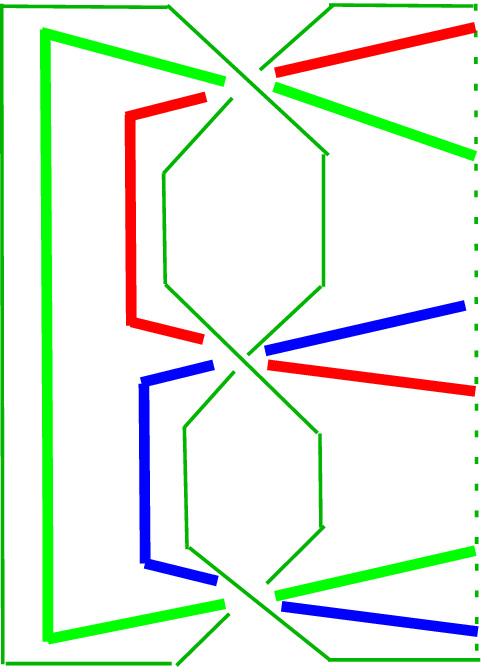}
   \label{fig:posthexsquare2}
}
\caption{}
\label{fig:posthexsquare12}
\end{figure}

\begin{figure}
\centering
\subfigure[]%caption in brackets
{
   \includegraphics[scale=0.55]{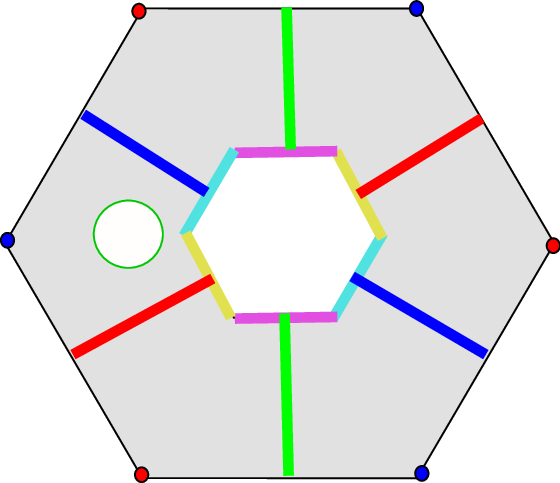}
   \label{fig:Qmonarcs}
}
\subfigure[]%caption in brackets
{
  \includegraphics[scale=0.4]{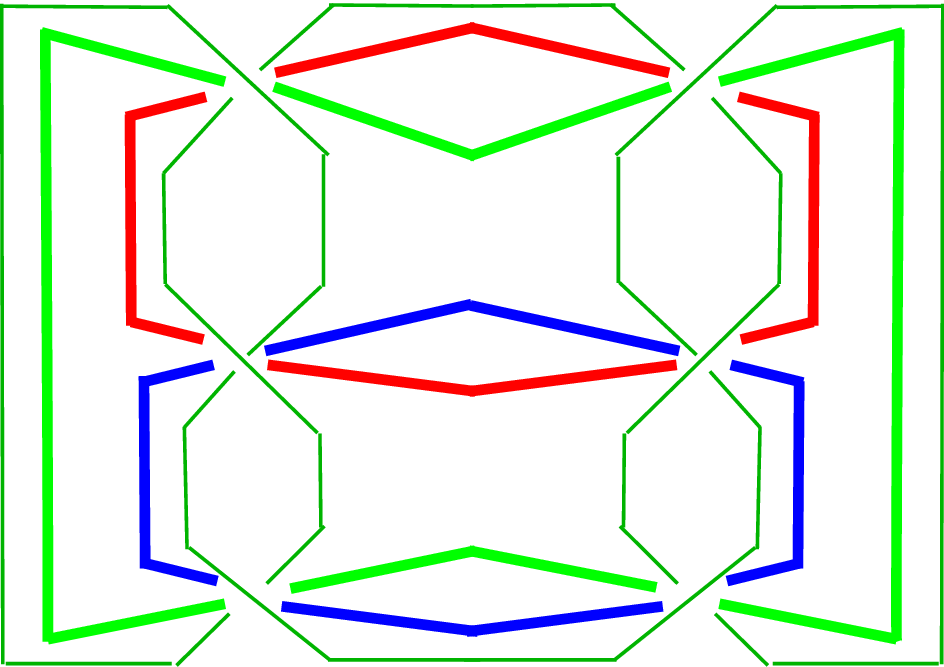}
   \label{fig:squaremon}
}
\caption{}
\label{fig:compare}
\end{figure}

 Note that the three circles in $F$ bear a striking resemblance to the two gray annuli and the red circle of Figure \ref{fig:prehexsquare5}.  We can exploit the resemblance to give a color-coded description of how the apparatus $A$ defining $V_n$ can be viewed in the hexagon model of the monodromy of $F_\cup$.  Picture $S^3 - Q$ as the hexagonal picture of $F \times I$ with  the top $F \times \{1\}$ identified to the bottom $F \times \{0\}$ by the monodromy twist.  Imagine how $A$ would look, viewed from above (i. e. looking down at the top $F \times \{1\}$) after coloring the upper gray annulus in Figure \ref{fig:prehexsquare6} red, the lower gray annulus blue, and the connecting arcs yellow.  This is shown in Figure \ref{fig:hexp1}: away from $\bdd F \times I$ the annuli lie near $F \times \{\frac12\}$.   As the left ends of both the red and blue annuli are incident to $\bdd F \times \{\frac12\}$ they rise along  $\bdd F \times I$ until they pass out of the top $F \times \{1\}$ (represented by a green dot) and continue their rise from the bottom $F \times \{0\}$ until they reach $\bdd F \times \{\frac12\}$ again and are joined to the rest of the red and blue annuli there.  The switch in perspectives (the annuli climbing the vertical wall $\bdd F \times I$ rather than circling around the knot $Q$) changes the apparent sign of the half-twist at $\bdd F \times \{\frac12\}$ from $\pm \frac12$ in Figure  \ref{fig:prehexsquare6} to $\mp \frac12$ in Figure  \ref{fig:hexp1}, much as the apparent half-twist on a ribbon in a book cover will change sign when the book is fully opened (see Figure \ref{fig:bookopen}).  The upshot is that the apparent framing of both the blue and red annuli in Figure  \ref{fig:hexp1} is now zero, the ``blackboard framing".

\begin{figure}
 \centering
    \includegraphics[scale=0.5]{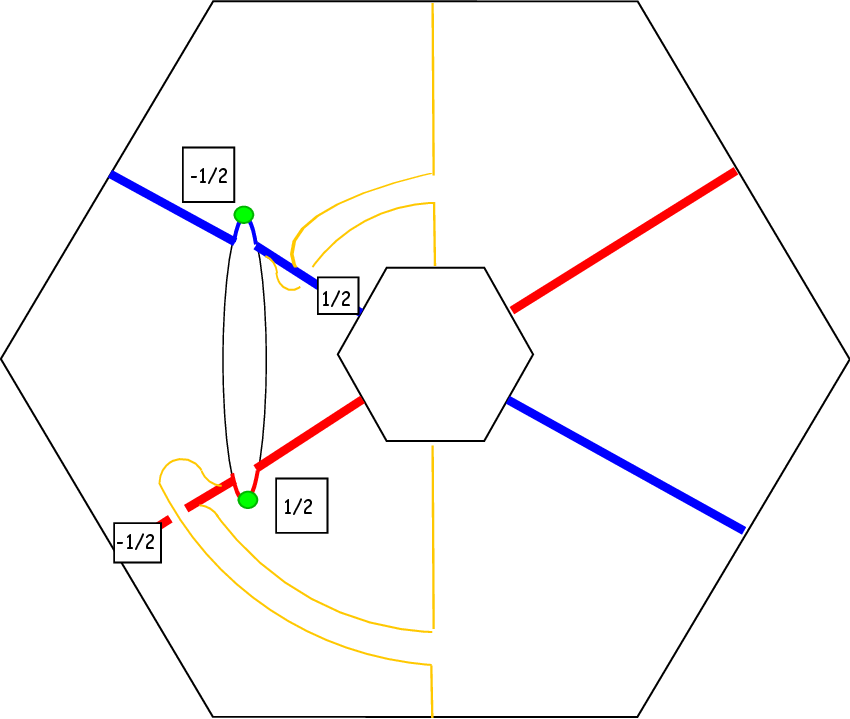}
    \caption{} \label{fig:hexp1}
    \end{figure}
    
    \begin{figure}
 \centering
 \labellist
\small\hair 2pt
\pinlabel  $+\frac12\;$twist at 50 -2
\pinlabel  $-\frac12\;$twist at 150 -2
\endlabellist
    \includegraphics[scale=1]{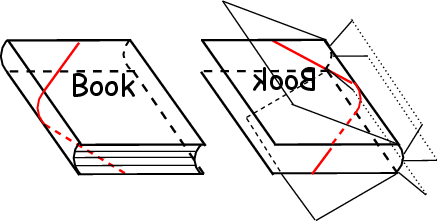}
    \caption{} \label{fig:bookopen}
    \end{figure}

Begin a process of pushing $V_n$ in $M$ until all of $V_n$ lies at the level $F_\cup \times \{\frac12\}$.  First compare Figures \ref{fig:hexp1} and \ref{fig:hexp2}:  Since the isotopy will be in $M$, $F \times I$ has been replaced by $F_\cup \times I$, filling in the missing disk, with only the two green dots remaining.  The green dots continue to represent the points at which the piece of the red and blue annuli to the right of the dots emerge from the top of the box $F_\cup \times I$ and, simultaneously, where the red and blue annuli to the left of the dots enter the bottom of $F_\cup \times I$. (A vertical cross-section of the northwest sextant, roughly parallel to the blue annulus, appears in Figure \ref{fig:xsection}.) One of the two arcs in $A$ that connects the red annulus to the blue annulus has changed color from yellow to brown. This will be the arc $\beta$ that is pushed first through the top of the box.  The point at which $\beta$ is incident to the blue annulus has been slid in $F_\cup \times \{\frac12\}$ so that it lies just to the left of the green dot instead of to the right.  Now push $\beta$ through the top of the box, at which point it reappears in $M$ at the bottom of the box, but, because of the monodromy, with a $\frac{\pi}3$ clockwise rotation of the ends of $\beta$ on the inner and outer hexagons.  The result, after also sliding the end of $\beta$ on the red annulus back to its original position to the left of the green dot, is shown in Figure \ref{fig:hexp3}.  The movement of the end of $\beta$ at the blue annulus is shown schematically, clockwise from the upper left, in Figure \ref{fig:xsection}.

\begin{figure}
 \centering
    \includegraphics[scale=0.5]{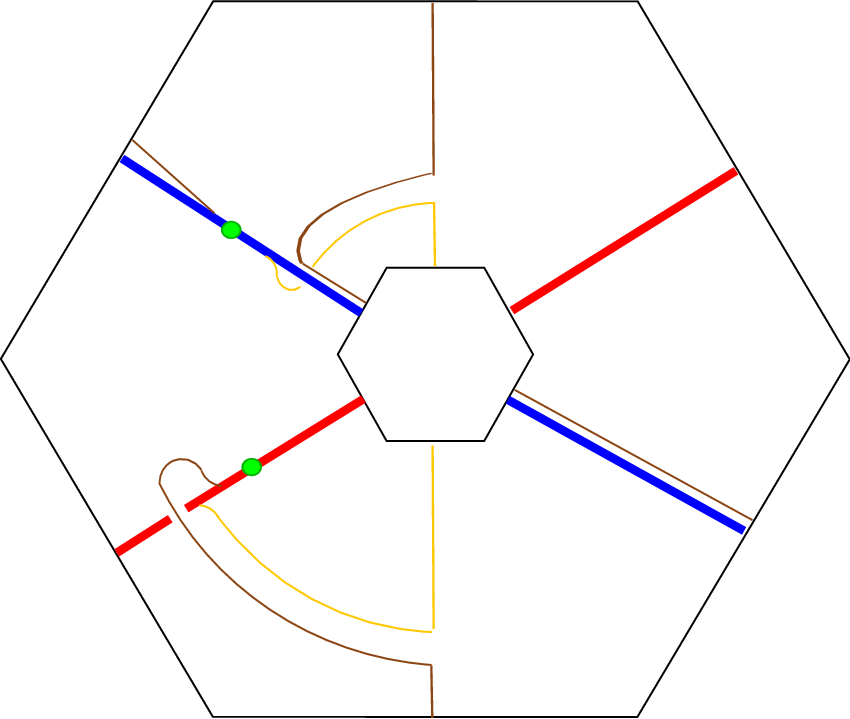}
    \caption{} \label{fig:hexp2}
    \end{figure}

\begin{figure}
 \centering
    \includegraphics[scale=0.5]{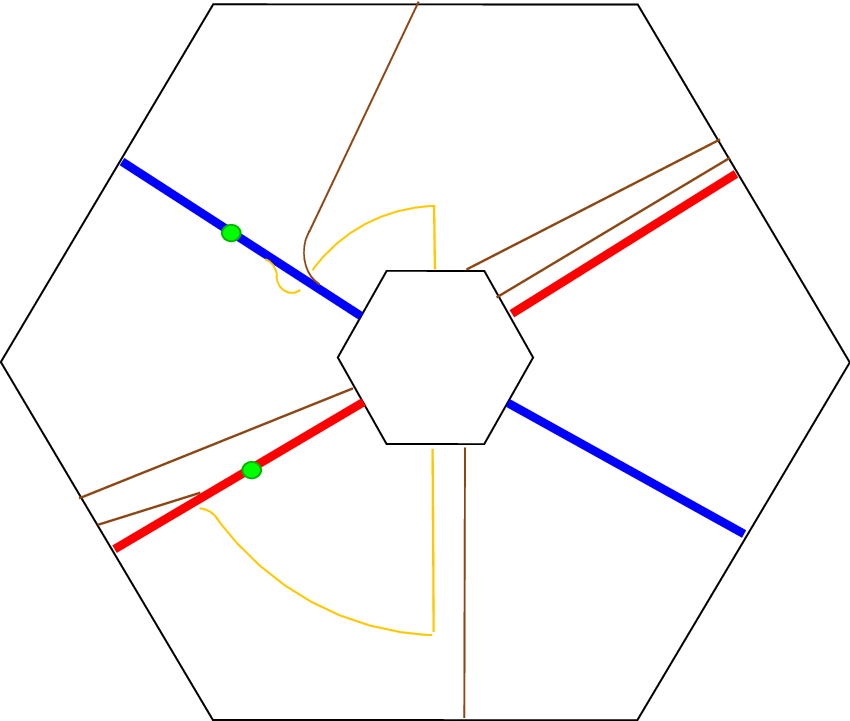}
    \caption{} \label{fig:hexp3}
    \end{figure}

\begin{figure}
 \centering
    \includegraphics[scale=0.5]{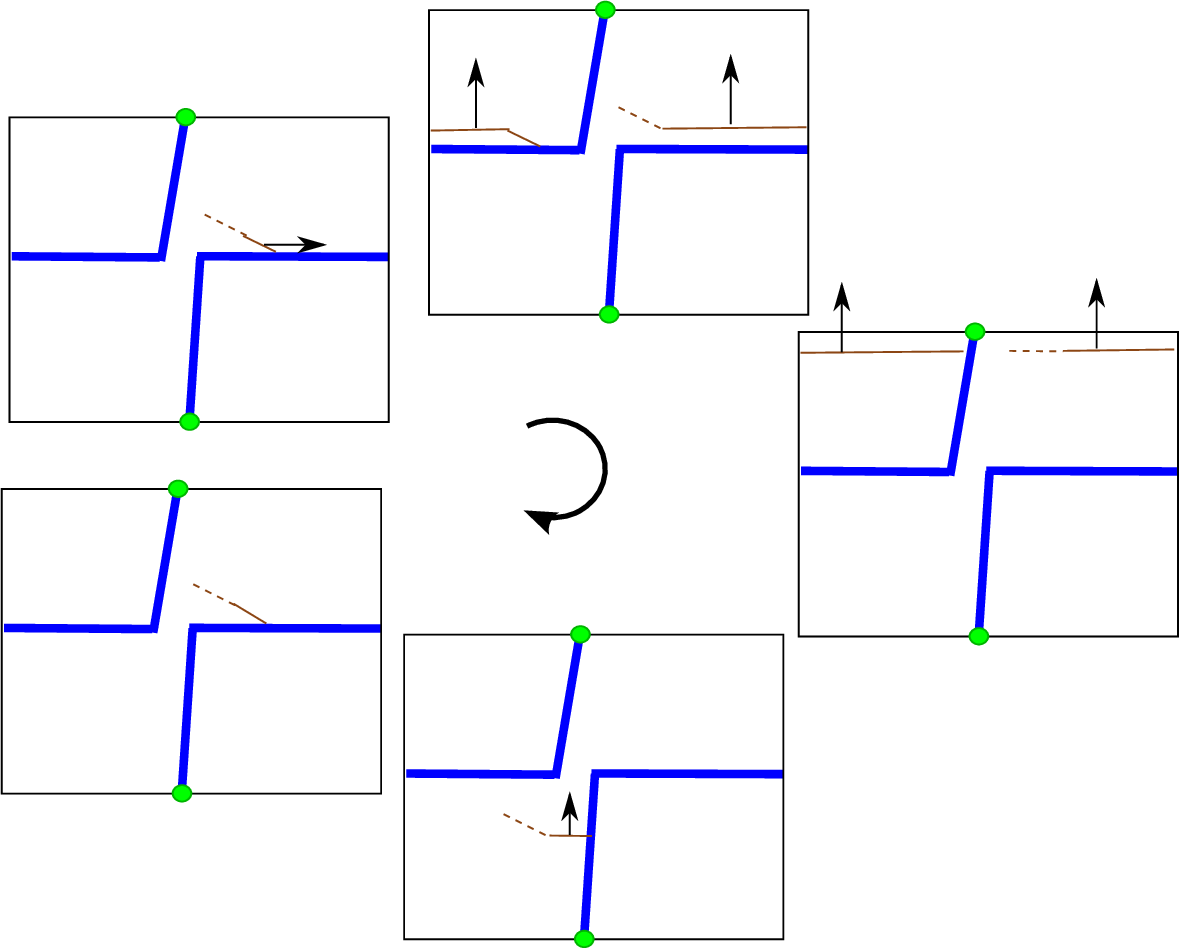}
    \caption{} \label{fig:xsection}
    \end{figure}

A symmetric argument describes how to push {\em down} the other arc (shown in yellow) that connects the blue and red annuli.  The result is shown in two steps in Figures \ref{fig:hexp4} and \ref{fig:hexp5}.  The more pleasing Figure \ref{fig:hexp6} is then obtained by sliding the points where the arcs are incident to the annuli, so that they all appear on the right side of the figure.  The red and blue annuli are now only $n-2$ arcs wide, so in the case $n = 2$, Figure \ref{fig:hexp6}, with the red and blue annuli omitted, represents the final positioning of $V_2$, completely on the surface $F_\cup$ as was predicted to be possible.  

\begin{figure}
\centering
\subfigure[]%caption in brackets
{
   \includegraphics[scale=0.4]{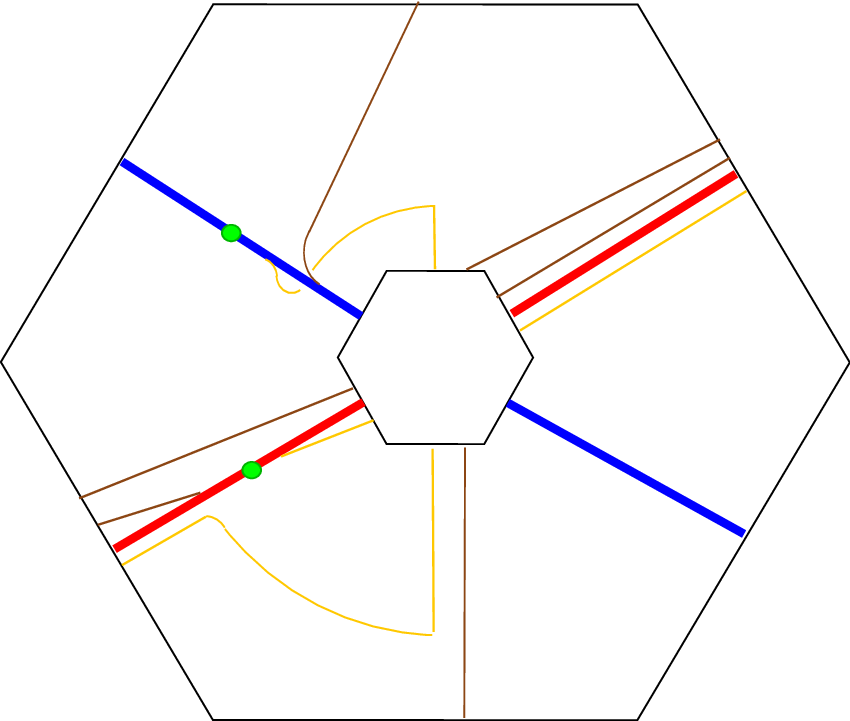}
   \label{fig:hexp4}
}
\subfigure[]%caption in brackets
{
  \includegraphics[scale=0.4]{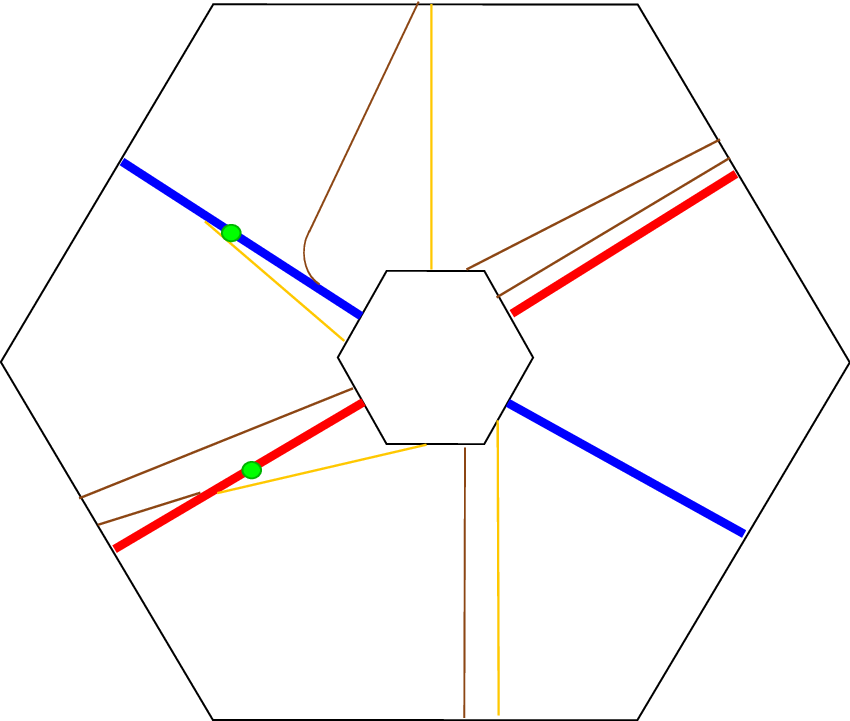}
   \label{fig:hexp5}
}
\subfigure[]%caption in brackets
{
  \includegraphics[scale=0.5]{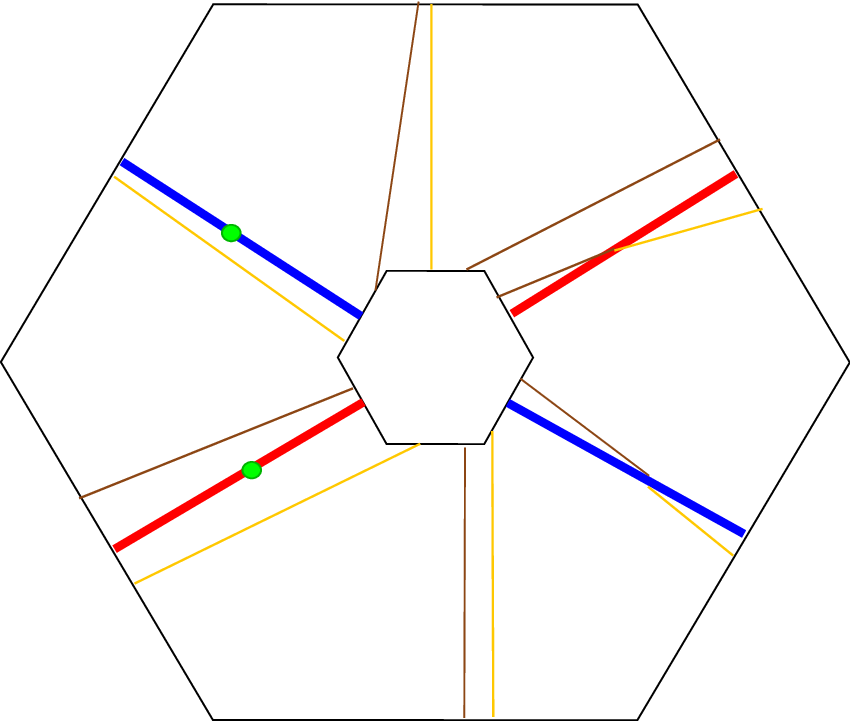}
   \label{fig:hexp6}
}
\caption{}
\label{fig:pushdown}
\end{figure}

\begin{figure}
\centering
\subfigure[]%caption in brackets
{
   \includegraphics[scale=0.41]{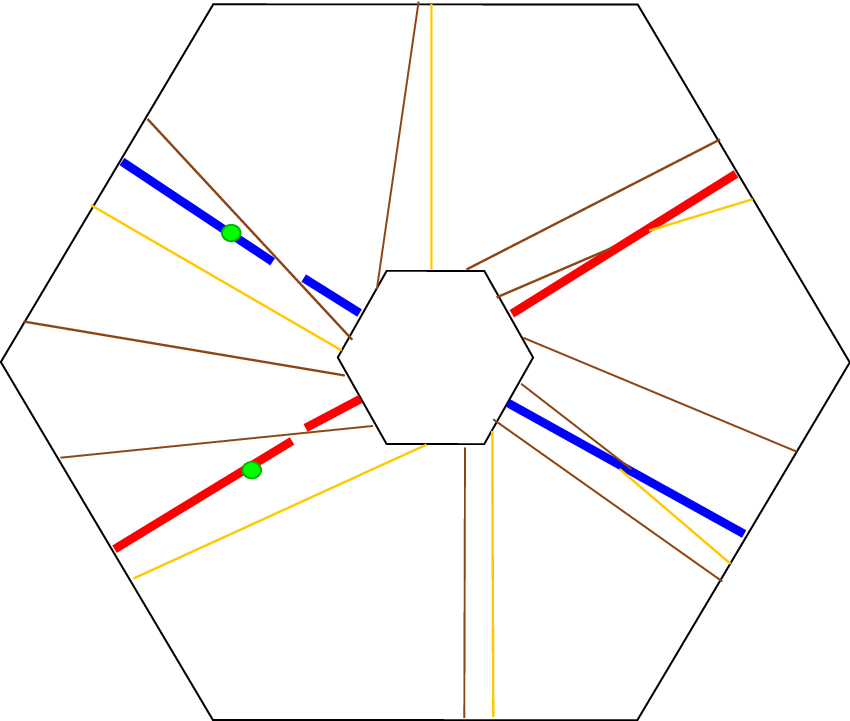}
   \label{fig:hexp8}
}
\subfigure[]%caption in brackets
{
  \includegraphics[scale=0.41]{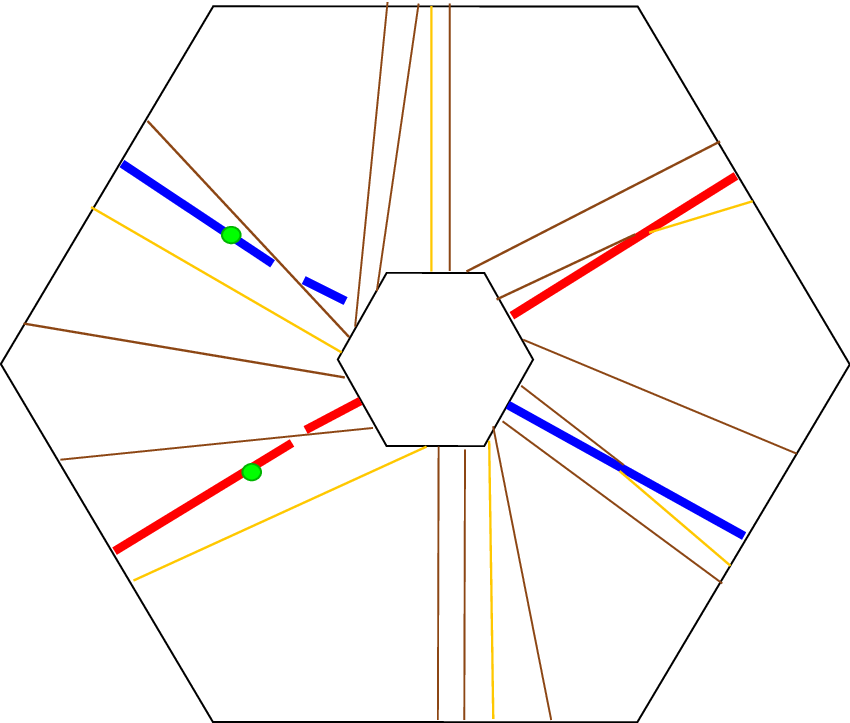}
   \label{fig:hexp10}
}
\caption{}
\label{fig:twopush}
\end{figure}

\begin{figure}
\centering
\subfigure[$n=3$]%caption in brackets
{
   \includegraphics[scale=0.41]{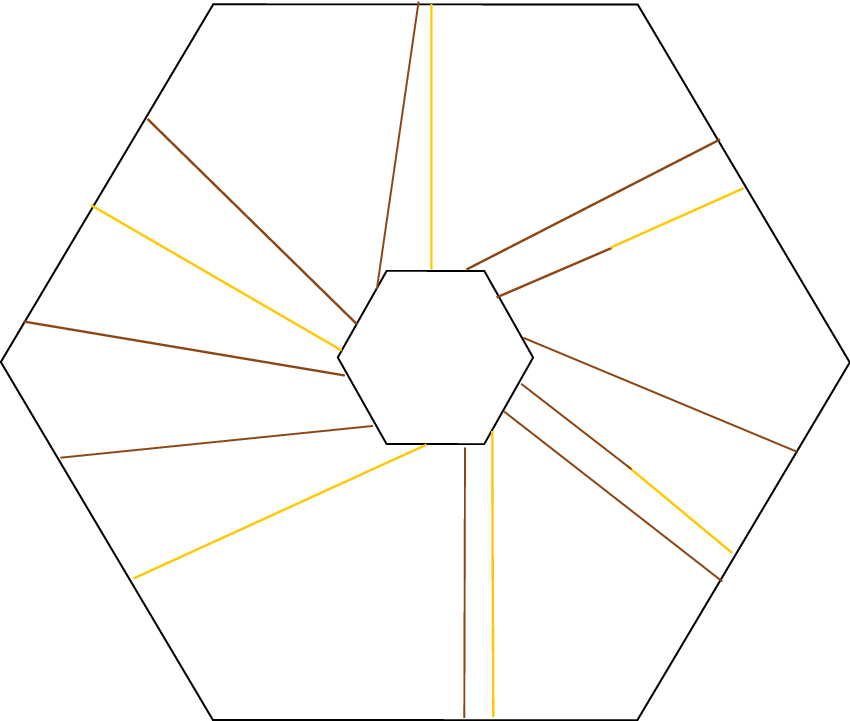}
   \label{fig:deg7ex}
}
\subfigure[$n=4$]%caption in brackets
{
  \includegraphics[scale=0.41]{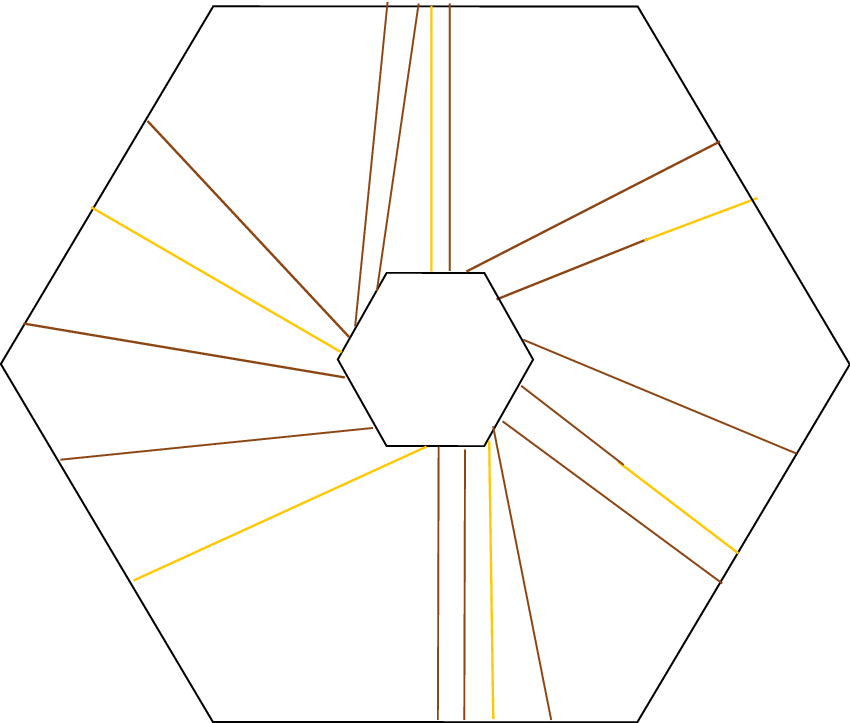}
   \label{fig:deg9ex}
}
\caption{}
\label{fig:n=34}
\end{figure}

If $n \geq 3$ further moves of $V_n$ are required.   Figures \ref{fig:hexp8} and \ref{fig:hexp10} show two more upward pushes of the brown arc $\beta$ through the top of the box $F_\cup \times I$ and, via the clockwise $\frac{\pi}3$ rotation of the monodromy, back through the bottom of the box. (Figures \ref{fig:deg7ex} and \ref{fig:deg9ex} show the corresponding final position of $V_n, n = 3, 4$ on a fiber $F_\cup$ of $M$.)    The fact that, after the push, a segment of $\beta$ crosses {\em over} over the blue annulus (and another crosses over the red annulus) is at first puzzling.  But recall how the monodromy acts on $F_\cup$:  Up to isotopy, it is a $\frac{\pi}3$ clockwise twist, but it also fixes the green dots in the figure (the points where the blue and red annuli intersect the top and the bottom of the box $F_\cup \times I$.  So the monodromy itself looks a bit like that in Figure \ref{fig:Qmon}.  Thus a segment of $\beta$ spanning the southern sextant of $F$, when pushed out the top of $F \times I$ and reappearing at the bottom, has its ends rotated to the southwest sextant, but the middle of the segment will pass south of the green dot on the red annulus.  The segment does then pass under the red annulus, but the segment can be straightened by an isotopy that appears to move the segment from below the annulus to above the annulus.  See Figure \ref{fig:straighten}, which also shows the isotopy in a vertical cross-section near the red annulus in the southwest sextant.  Similar remarks apply to segments of $\beta$ passing over the blue annulus. 

\begin{figure}
 \centering
    \includegraphics[scale=0.7]{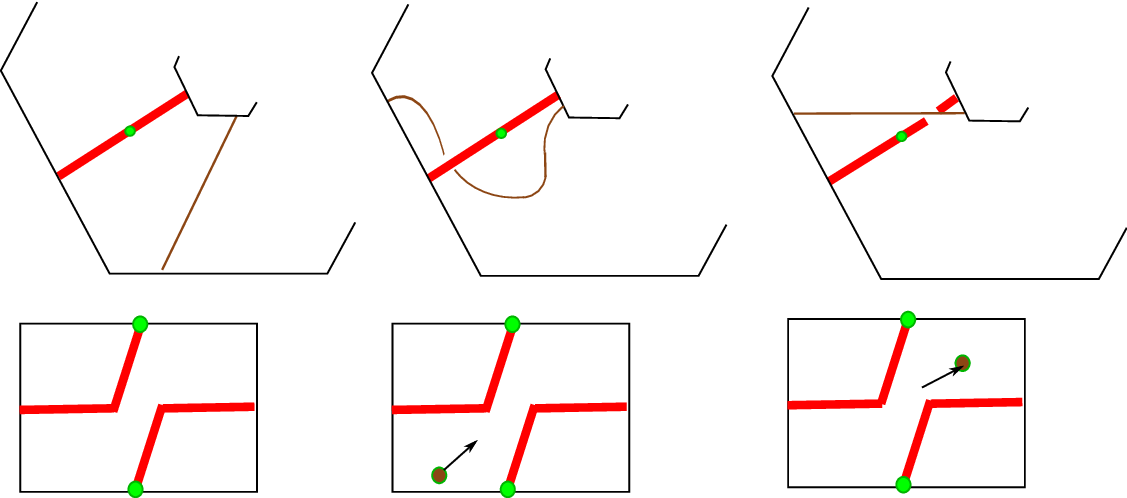}
    \caption{} \label{fig:straighten}
 \end{figure}

The final appearance of $V_n$ on the fiber of $M$ depends mostly on $n \; mod\;3$.  Figure \ref{fig:finalrest} shows the general case for $n \equiv -1, 0, 1$ by depicting with brown bands collections of $j-1$ parallel segments of $\beta$.    The blue and red annuli in the figures can be ignored; they have been included only to help imagine the transition from one step to the next. At several places in the figure it appears that a single segment of $\beta$ intersects a brown band, but this is just shorthand for a double-curve sum of the crossing arc with the $j-1$ curves in the band, as shown in Figure \ref{fig:hexgenRem2a}.) 

\begin{figure}
\centering
\labellist
\small\hair 2pt
\pinlabel  $j-1$ at -25 140
\endlabellist
\subfigure[$n=3j-1$]%caption in brackets
{
   \includegraphics[scale=0.51]{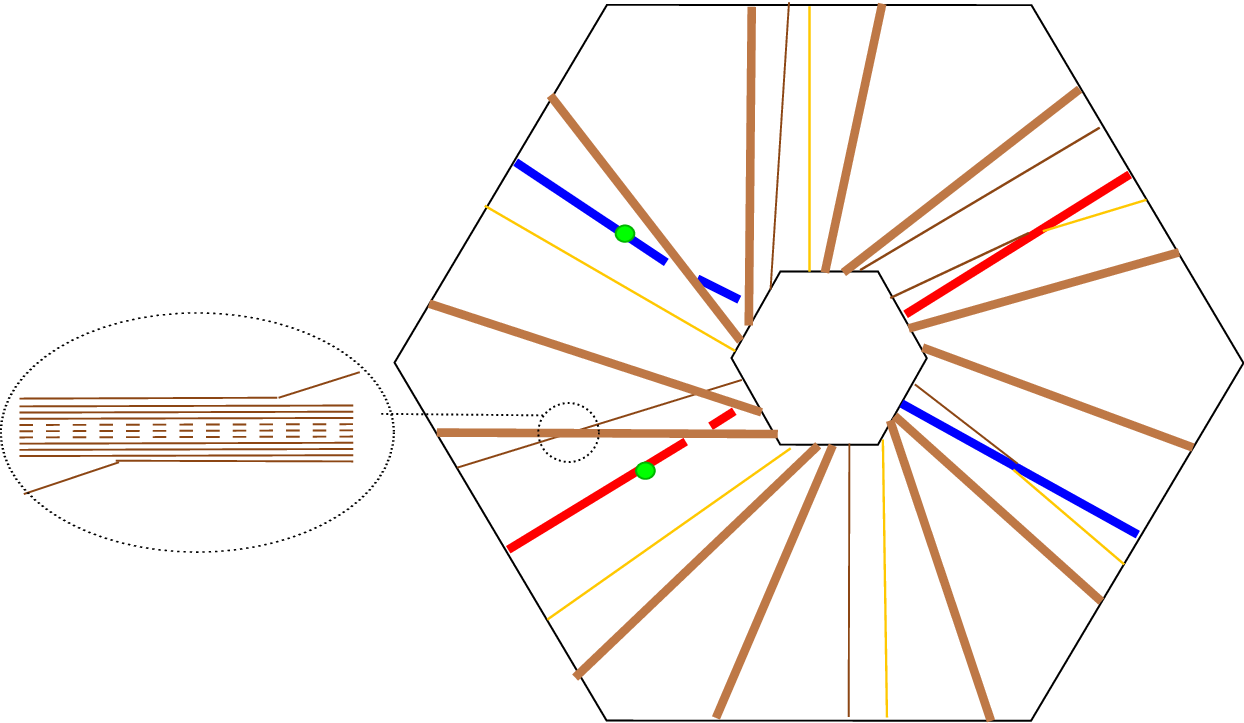}
   \label{fig:hexgenRem2a}
}
\subfigure[$n=3j$]%caption in brackets
{
  \includegraphics[scale=0.41]{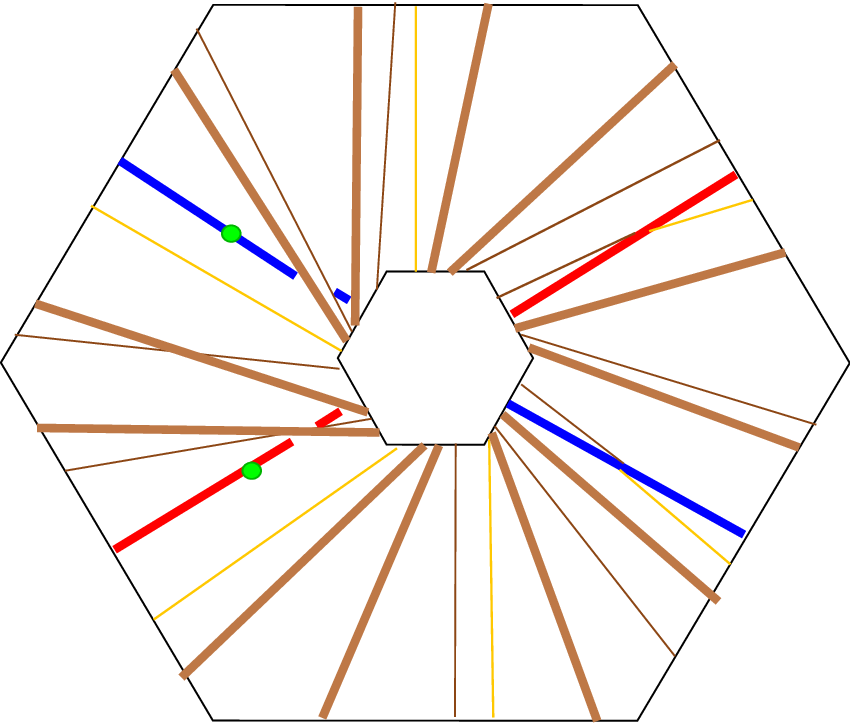}
   \label{fig:hexgenRem0}
}
\subfigure[$n=3j+1$]%caption in brackets
{
  \includegraphics[scale=0.41]{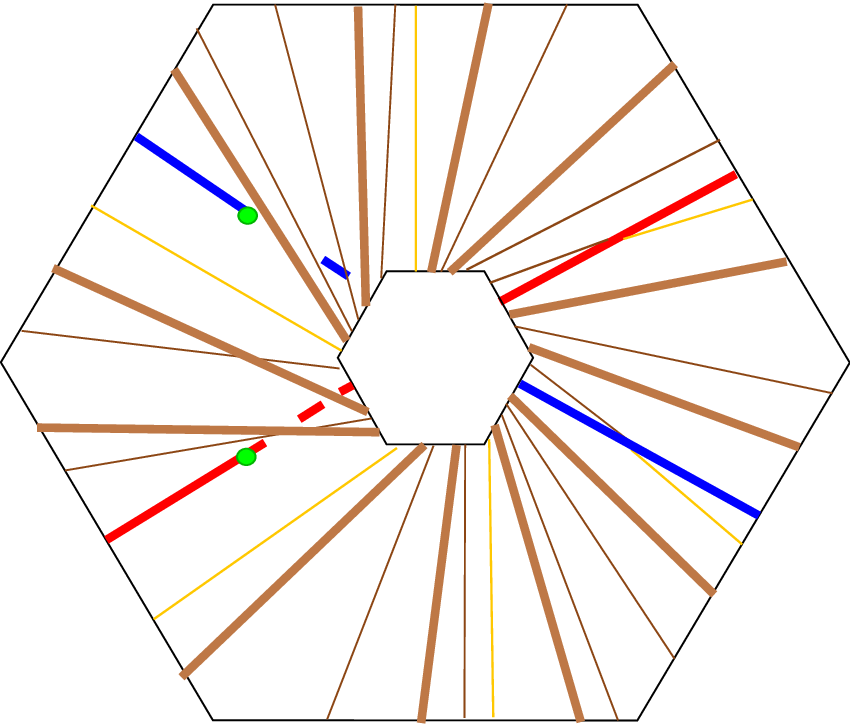}
   \label{fig:hexgenRem1}
}
\caption{}
\label{fig:finalrest}
\end{figure}

Of course all these presentations of $V_n \subset F_\cup \subset M$ can be translated to pictures of $V_n \subset F \subset S^3$. This translation, in the case $n=3$, is shown in Figure \ref{fig:ribbon} via a three-stage process:  The points in which $V_n$ intersects the circle separating the trefoil summands are labeled sequentially, as shown in Figure \ref{fig:hexribbon_n=3a}.  Then the discussion around Figure \ref{fig:posthexsquare12} is used to locate each of these subarcs in the appropriate place on the Seifert surface $F \subset S^3$, as shown in Figure \ref{fig:sqribbon_n=3a}. These subarcs are mostly joined along the arc that separates the left- and right-hand trefoil knots, but there is a dangling end at both the top and the bottom of $F$ in the figure, reflecting that joining these ends by a subarc requires a choice of how $V_n$ is to avoid the disk $F_\cup - F$ bounded by $Q$.  The choice is whether to connect these ends by an arc parallel to the trefoil knot on the left or parallel to the one on the right.  Figure  \ref{fig:sqribbon_n=3b} shows the result when the two ends are connected along the trefoil on the left, via an arc that is rendered in red.   The resulting link in $S^3$ is slice by construction; it is a nice question (just as it was for \cite[Figure 2]{GST}) whether this link, or examples from higher $n$, are also ribbon links.  Alex Zupan \cite{Zu} has pointed out that this question is of more than passing interest:  it is easy to see that handle slides preserve the property of being a ribbon link, so if these examples are not ribbon links then they are also counterexamples to Property 2R.  

\begin{figure}
\centering
\subfigure[]%caption in brackets
{
   \includegraphics[scale=0.5]{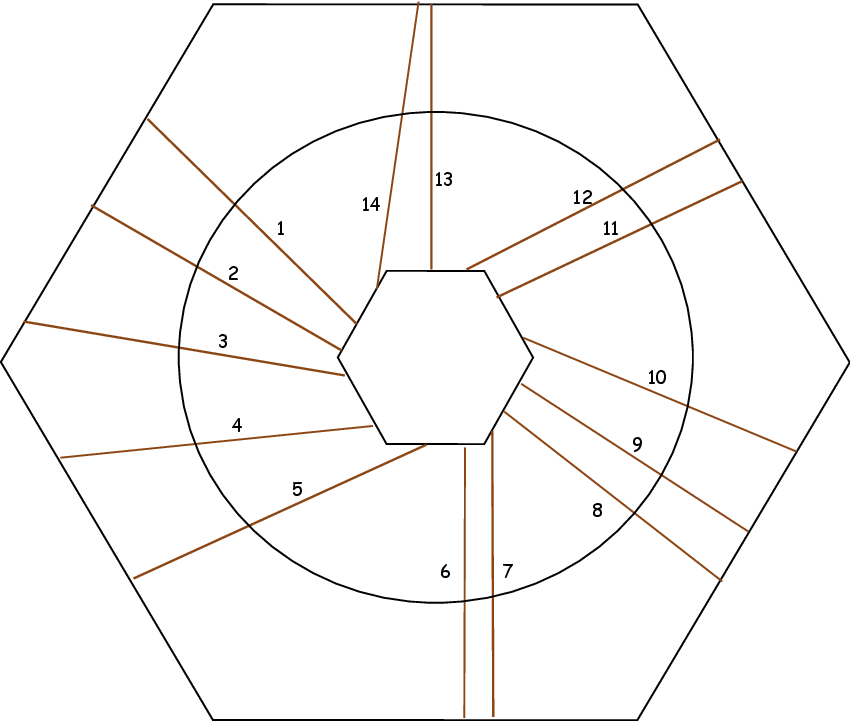}
   \label{fig:hexribbon_n=3a}
}
\subfigure[]%caption in brackets
{
  \includegraphics[scale=0.5]{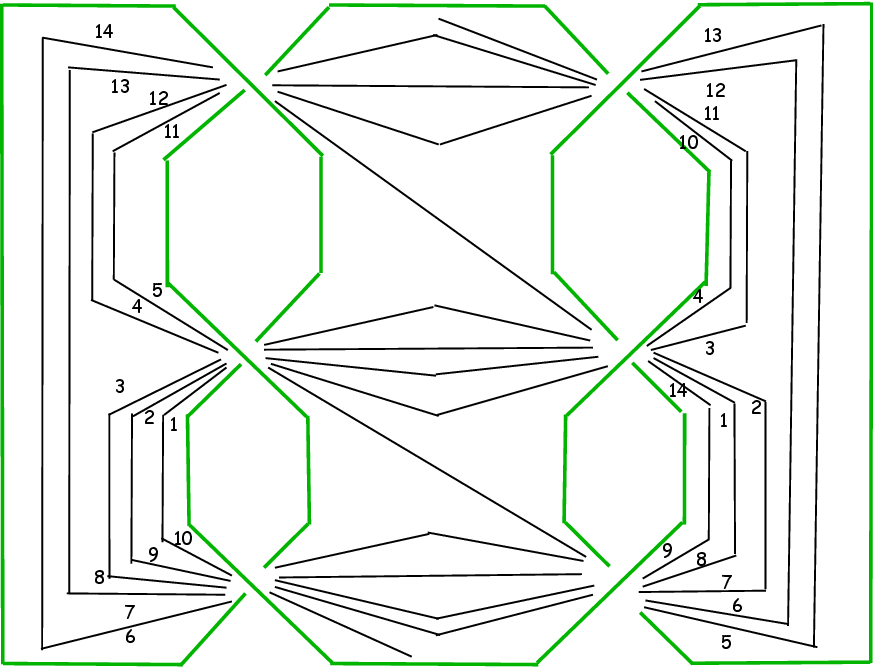}
   \label{fig:sqribbon_n=3a}
}
\subfigure[Is it a ribbon link?]%caption in brackets
{
  \includegraphics[scale=0.5]{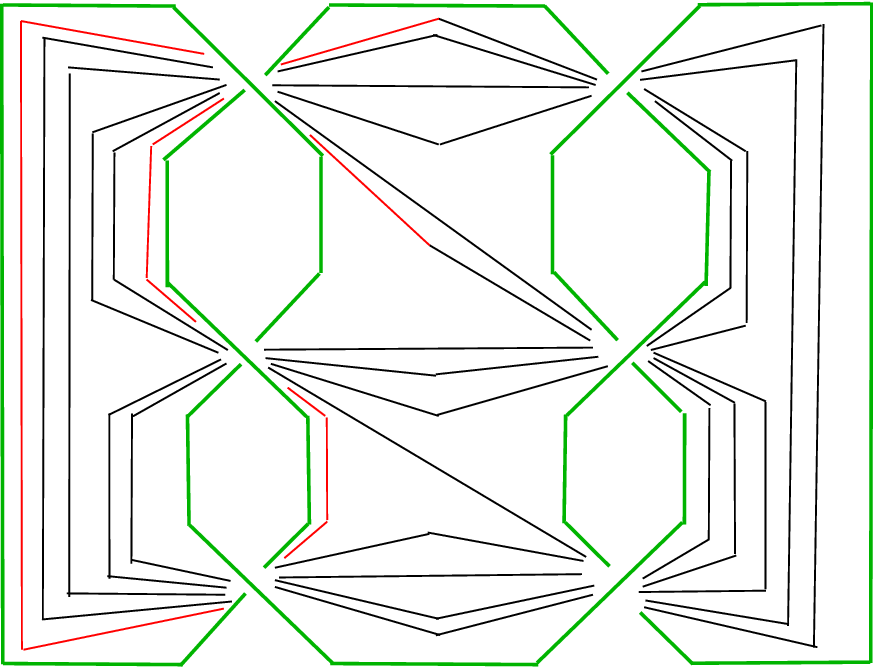}
   \label{fig:sqribbon_n=3b}
}
\caption{}
\label{fig:ribbon}
\end{figure}

\section{The curve that $V_n$ covers in a 4-punctured sphere} \label{sect:classification}

Recall from \cite{GST} that there is a natural $\mathbb{Z}_3$ action on the genus $2$ surface $F_\cup$, by which $F_\cup$ is a $3$-bold branched cover of $S^2$ with $4$ branch points. See \cite[Figure 7]{GST}, reproduced as Figure \ref{fig:rhosigma} here.  It is further shown (\cite[Corollaries 6.2, 6.4]{GST}) that $0$-framed surgery on a simple closed curve $V \subset F_\cup \subset M$ yields $(S^2 \times S^2) \# (S^2 \times S^2)$ if and only if $V$ projects homeomorphically to an essential simple closed curve in the $4$-punctured sphere $P = S^2 - \{branch\;points\}$ and that an essential simple closed curves in $P$ is the homeomorphic image of a curve in $F_\cup$ if and only if it separates two branch points coming from the same trefoil summand of $Q$.

\begin{figure}
 \centering
    \includegraphics[scale=0.6]{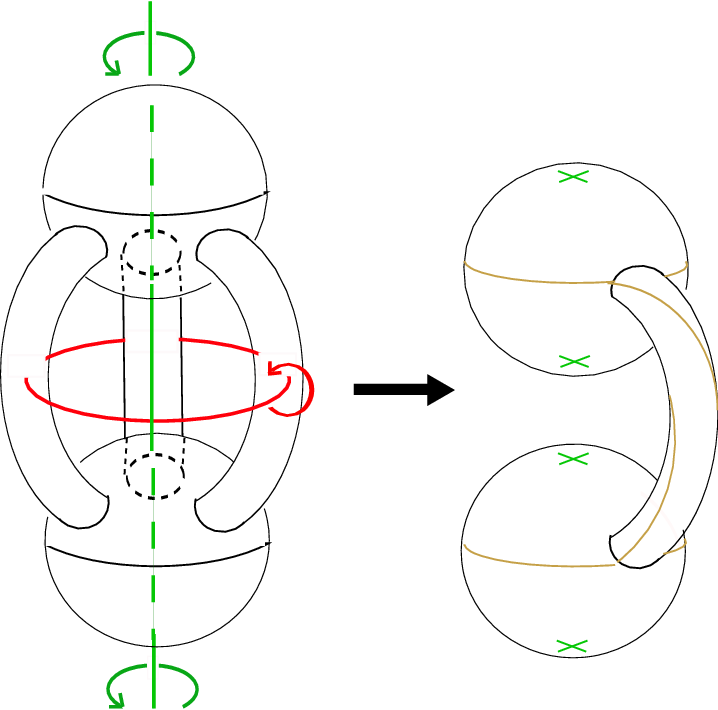}
    \caption{} \label{fig:rhosigma}
 \end{figure}

The hexagonal description of $F_\cup$ given in Section \ref{sect:hex} (e.g. Figure \ref{fig:Mmon}) is particularly easy to see as a branched cover over $S^2$.  Start by giving $S^2$ the ``pillowcase" metric: view $S^2$ as constructed from two congruent rectangles, with their boundaries identified in the obvious way.  The corners of the rectangles will be the branch points for the covering and the two rectangles will be called the front face and the back face of $P \subset S^2$.   Identify the top sextant of Figure  \ref{fig:Mmon} with the front face of $P$ and wrap the other five sextants equatorially around $P$.   See  Figure \ref{fig:prebridge}.
Then the northeast, northwest and the bottom sextant are all identified with the back face of $P$ and the southeast and southwest sextants with the front face.  The identifications of the boundary edges in Figure \ref{fig:Mmon} are consistent: for example the top and bottom edges of the outside hexagon in the figure have been identified with the top edge of, respectively, the front and the back face of $P$, so the identification of these edges to create $F_\cup$ is consistent with the identification of the top edges of the two rectangles to form $S^2$.

\begin{figure}
 \centering
    \includegraphics[scale=0.4]{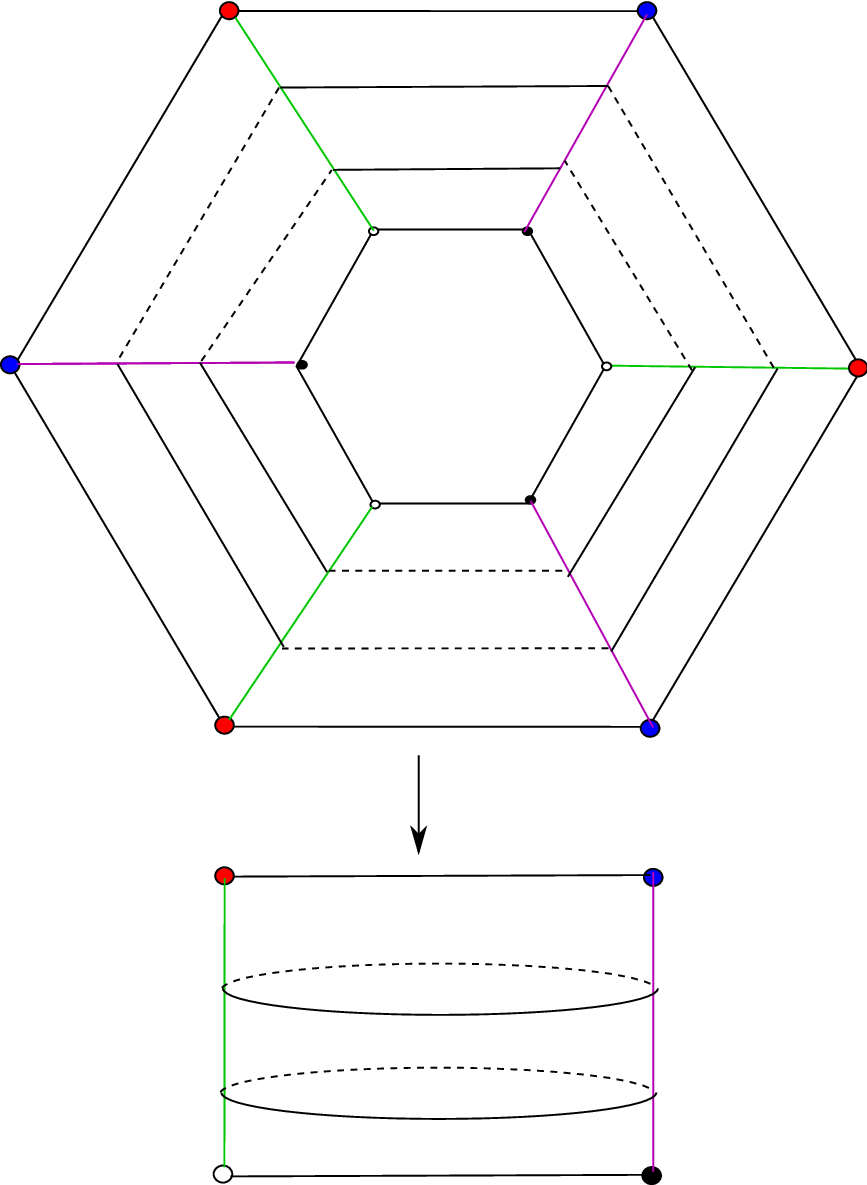}
    \caption{} \label{fig:prebridge}
 \end{figure}

There is a natural correspondence between isotopy classes of simple closed curves in $P$ and the extended rationals $ \mathbb{Q} \cup \infty$.  The correspondence is given by the slope in the pillowcase metric.  It is a bit more useful in our context to take the reciprocal of the apparent slope in Figure \ref{fig:prebridge} or, equivalently, to turn the pillowcase in the figure on its side. Thus one of the horizontal curves in $P$ shown in Figure \ref{fig:prebridge} has slope $0$, so it will correspond here to $\frac 10 = \infty \in \mathbb{Q}  \cup \infty$.  Such a curve is $3$-fold covered by a simple closed curve in $F_\cup$ that divides $F_\cup$ into the two genus-one 
surfaces $F_\ell$ and $F_r$, Seifert surfaces for the two trefoil summands of $Q$.  A simple closed curve in $P$ separates the two 
branch points lying in $F_\ell$ (equivalently, separates the two that lie in $F_r$) if and only if it intersects the top seam of the 
pillowcase in an odd number of points; that is, if and only if it corresponds to a fraction $\frac pq \in \mathbb{Q} \cup \infty$ for which 
$q$ is odd. So, for example, a circle in $P$ that is vertical in Figure \ref{fig:prebridge} is assigned $0 = \frac 01 \in \mathbb{Q}$ and corresponds to any of the three unknotted circles in $F$ shown in Figure \ref{fig:squaremon}.  In this manner, $\{ p/q \in \mathbb{Q} \; | 
\; q \;\;odd \}$ becomes a natural index for the set of curves $V \subset S^3 - Q$ such that surgery on $Q \cup V$ gives $(S^1 \times 
S^2) \# (S^1 \times S^2)$.  (Since $\infty = \frac 10$ does not have odd denominator, we can ignore it.)  It is natural to ask exactly how 
the curves $V_n$ fit into this classification scheme.  

\bigskip

\noindent {\bf Remark:}  Here is the rationale for taking the reciprocal of the apparent slope, i. e. for turning Figure \ref{fig:prebridge} on its side.  There is a natural automorphism $(S^3, Q) \to (S^3, Q)$ called a {\em twist} (see \cite{Z} for a related use of the term).   A twist fixes one of the trefoil summands of $Q$ and rotates the other trefoil summand fully around the two points at which the summands are joined.  (A more technical description:  a twist is a meridional Dehn twist along a swallow-follow companion torus $T_{sf}$ for $Q$.)  It's fairly easy to see that a twist changes a curve $V \subset F$ indexed by $\frac pq \in \mathbb{Q} \cup \infty$ to one  indexed by  $\frac {p \pm q}{q} \in \mathbb{Q} \cup \infty$.  So if we also allow $V$ to change by such twists of $Q$, which do not change the isotopy class of the link $Q \cup V$, as well as by isotopy and slides of $V$ over $Q$, we could even index the curves $V$ by $\frac pq \in \mathbb{Q/Z}$, $q$ odd and, as a result, focus attention on those indices $\frac pq$ in which $|2p| < q$.  The next theorem shows that, from this point of view, the curves $V_n$ are at the extreme.

\begin{thm} \label{thm:class} In the classification scheme above, the curve $V_n$ corresponds to $\frac{n}{2n+1} \in \mathbb{Q}.$
\end{thm}

\begin{proof}  Examine Figure \ref{fig:finalrest}, ignoring the red and blue bands and recalling that each wide brown arc represents $j-1$ parallel arcs.  The number of intersection points of $V_n$ with the outer hexagon is the total number of arcs that appear, $4n + 2$.  The number of intersection points of $V_n$ with the six lines that divide the figure into sextants is $2n$.  The ratio is then $\frac{n}{2n+1}$.
\end{proof}

The case $n = 4$ is shown in Figure \ref{fig:bridge}.

\begin{figure}
 \centering
    \includegraphics[scale=0.5]{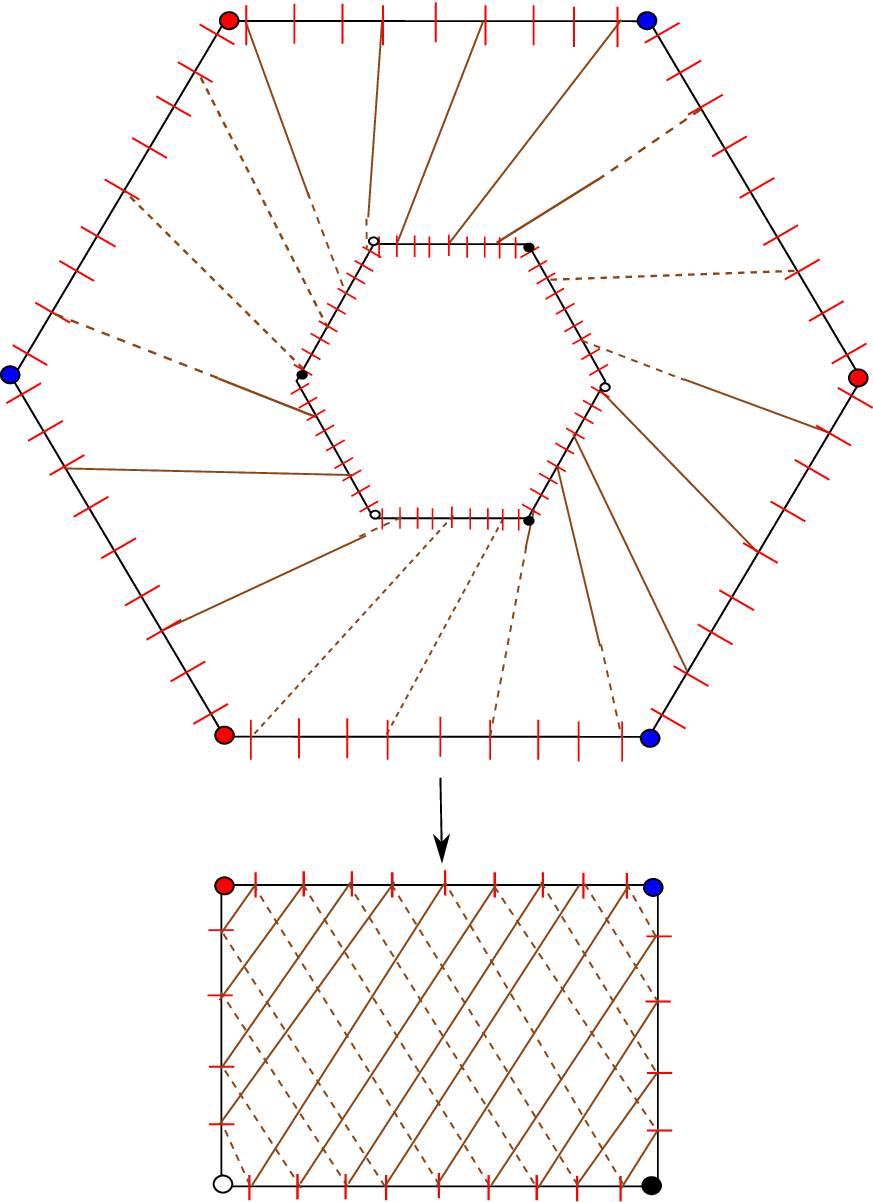}
    \caption{} \label{fig:bridge}
 \end{figure}

\section{An alternate view of the construction}\label{sect:alternate}

Essentially the first step in the construction above (see Figure \ref{fig:prehexstep2}) was to import \cite[Figure 12]{GST} and blow down the two $\pm 1$ bracketed unknots.  There is another way to organize the construction, one which delays the blow-down until much later and so gives additional insight into how $V_n$ lies in $M$.  Begin with the link diagram \cite[Figure 19]{GST} (a version of \cite[Figure 11]{GST}) that describes $L_n$.  This is shown here in Figure \ref{fig:GSTFig19}, augmented so that the relevant torus $T\subset M$ is more visible:  The vertical plane, mostly purple but containing a green disk, is a $2$-sphere in $S^3$ that contains a circle on which $0$-framed surgery is performed.  The surgery splits the $2$-sphere into a pair of $2$-spheres, one green and one purple.  In the diagram the two $2$-spheres are connected by two thick pink strands.  The torus $T$ is obtained by tubing the two $2$-spheres together along the annuli boundaries of the two pink strands.  The $\pm n$ twist-boxes represent the $n$-fold Dehn twist along the  meridian of $T$ used in the construction of $V_n$ (see \cite[Section 10]{GST}).  

\begin{figure}
     \labellist
\small\hair 2pt
\pinlabel  $n$ at 170 155
\pinlabel  $-n$ at 170 75
\pinlabel  $[1]$ at 25 275
\pinlabel  $[-1]$ at 223 275
\pinlabel  $0$ at 180 255
\pinlabel  $0$ at 25 155
\endlabellist
 \centering
    \includegraphics[scale=0.5]{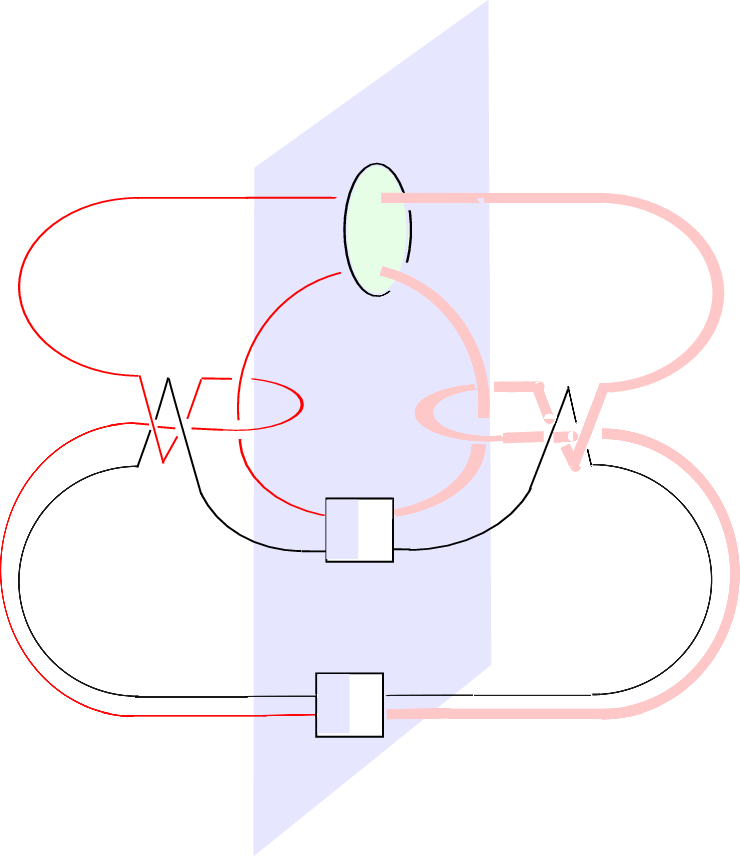}
    \caption{} \label{fig:GSTFig19}
    \end{figure}
    
Figure \ref{fig:simplify1} shows a sequence of isotopies which moves the link in Figure \ref{fig:GSTFig19} (including the bracketed unknots) to a position in which something like the square knot begins to appear.   When the top black circle is slid over the two red circles labeled $[\pm 1]$ the square knot fully emerges:  Figure \ref{fig:simplify2} shows the resulting circle as a green square knot, on which $0$-surgery is still to be performed, and also illustrates how the bracketed red circles can be pushed near its Seifert surface.  (The dotted parallel green and red arcs in the twist boxes are to indicate that the twisting is of the black curve around the red and green curves, not the red and green curves around each other.)
    
\begin{figure}  
\centering
\subfigure[]%caption in brackets
{
  \labellist
\small\hair 2pt
\pinlabel  $n$ at 172 100
\pinlabel  $-n$ at 168 20
\pinlabel  $[1]$ at 25 220
\pinlabel  $[-1]$ at 225 220
\pinlabel  $0$ at 180 200
\pinlabel  $0$ at 25 100
\endlabellist
   \includegraphics[scale=0.5]{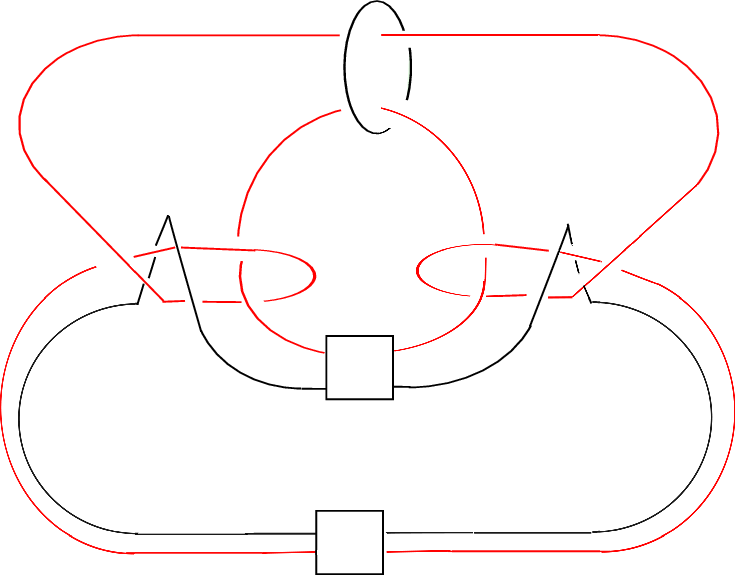}
   \label{fig:linksimplify1_5}
}
\subfigure[]%caption in brackets
{
  \labellist
\small\hair 2pt
\pinlabel  $-n$ at 165 100
\pinlabel  $n$ at 164 20
%\pinlabel  $[1]$ at 25 220
%\pinlabel  $[-1]$ at 220 220
%\pinlabel  $0$ at 180 200
%\pinlabel  $0$ at 25 100
\endlabellist
\includegraphics[scale=0.5]{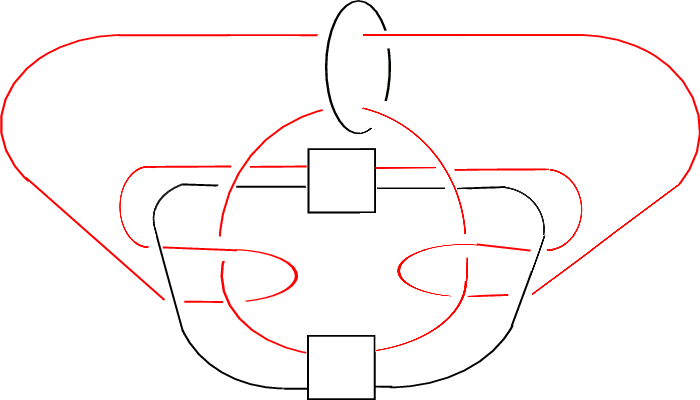}
   \label{fig:linksimplify1_6}
}
\subfigure[]%caption in brackets
{
  \labellist
\small\hair 2pt
\pinlabel  $-n$ at 165 100
\pinlabel  $n$ at 164 20
%\pinlabel  $[1]$ at 25 220
%\pinlabel  $[-1]$ at 220 220
%\pinlabel  $0$ at 180 200
%\pinlabel  $0$ at 25 100
\endlabellist
\includegraphics[scale=0.5]{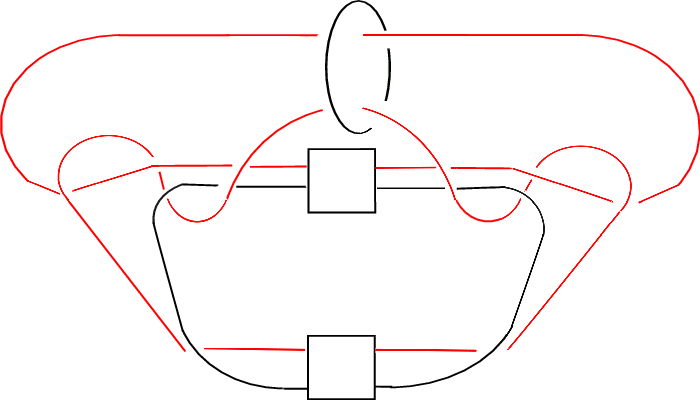}
   \label{fig:linksimplify1_7}
}
\subfigure[]%caption in brackets
{
  \labellist
\small\hair 2pt
\pinlabel  $-n$ at 130 100
\pinlabel  $n$ at 123 20
\pinlabel  $[1]$ at 275 160
\pinlabel  $[-1]$ at 275 80
\pinlabel  $0$ at 135 220
\pinlabel  $0$ at 50 80
\endlabellist
\includegraphics[scale=0.6]{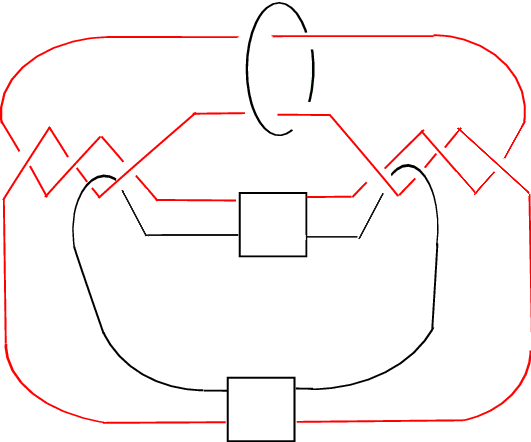}
   \label{fig:linksimplify4}
}
\caption{}
\label{fig:simplify1}
\end{figure}

\begin{figure}  
\centering
\subfigure[]%caption in brackets
{
  \labellist
\small\hair 2pt
\pinlabel  $-n$ at 130 105
\pinlabel  $n$ at 123 22
\pinlabel  $[1]$ at 275 160
\pinlabel  $[-1]$ at 275 80
%\pinlabel  $0$ at 135 220
%\pinlabel  $0$ at 50 80
\endlabellist
   \includegraphics[scale=0.8]{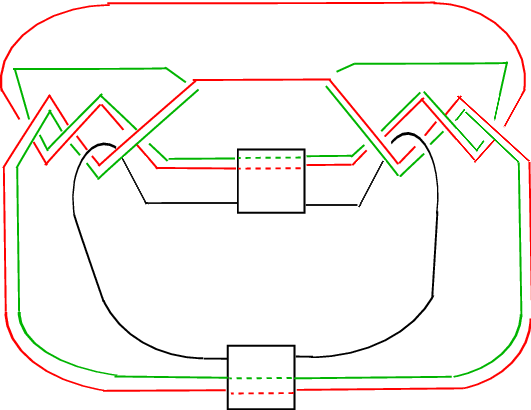}
   \label{fig:linksimplify5}
}
\subfigure[]%caption in brackets
{
  \labellist
\small\hair 2pt
\pinlabel  $-n$ at 130 105
\pinlabel  $n$ at 125 22
\pinlabel  $[1]$ at 125 50 
\pinlabel  $[-1]$ at 128 85
%\pinlabel  $0$ at 135 220
%\pinlabel  $0$ at 50 80
\endlabellist
\includegraphics[scale=0.8]{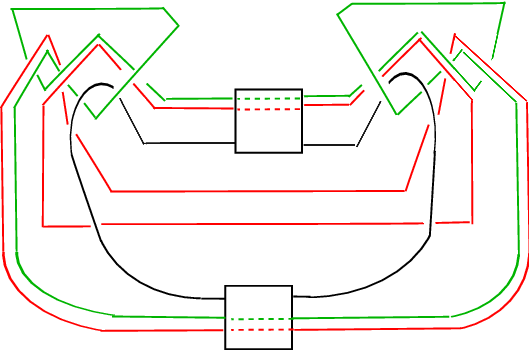}
   \label{fig:linksimplify5_2}
}
\caption{}
\label{fig:simplify2}
\end{figure}

In fact it is shown in Figure \ref{fig:simplify3} that, except for the twist boxes, both the red curves (labeled $[\pm 1]$) and the $0$-framed black curve can be simultaneously pushed onto the natural Seifert surface of the green square knot, all before the bracketed red curves are blown down!  (The apparent twist of the curves in Figure \ref{fig:linksimplify5_4} is canceled by a symmetric twist of the curves on the left side of the twist boxes.)

\begin{figure}  
\centering
\subfigure[]%caption in brackets
{
    \includegraphics[scale=0.8]{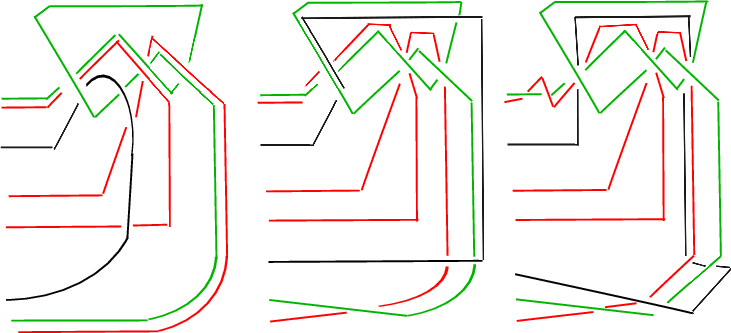}
   \label{fig:linksimplify5_4}
}
\subfigure[]%caption in brackets
{
  \labellist
\small\hair 2pt
\pinlabel  $-n$ at 125 105
\pinlabel  $n$ at 120 22
\pinlabel  $[1]$ at 125 50
\pinlabel  $[-1]$ at 123 85
\pinlabel  $0$ at 255 50
\pinlabel  $0$ at 215 50
\endlabellist
\includegraphics[scale=0.8]{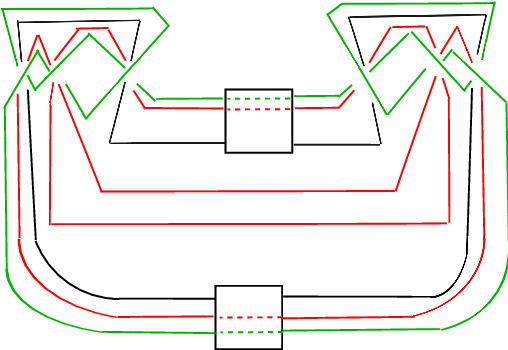}
   \label{fig:linksimplify6}
}
\caption{}
\label{fig:simplify3}
\end{figure}

When the square knot is straightened out, it appears as in Figure \ref{fig:linksimplify7}, which bears a striking resemblance to the earlier Figure \ref{fig:squaremon}.  In particular, if we temporarily ignore the $0$-surgered black curve (so we can also ignore the twist boxes), then the two red circles are parallel to each other in the {\em complement} of the square knot.  That is, there is an annulus $A \subset S^3 - Q$ whose boundary consists of the two red circles.  This can be seen directly in Figure \ref{fig:linksimplify7}, but it also follows from the discussion surrounding Figure \ref{fig:squaremon}: one red circle is the image under the monodromy of the other.  Since the red circles have opposite (bracketed) signs, it follows that blowing both of them down simultaneously has no effect on the square knot: it persists after the blow-down, but any arc that intersects the annulus $A$ between the red curves will be twisted around the core of $A$.  In particular, the $0$-surgered black curve intersects $A$ $n$ times at each of the upper and lower twist boxes, so the simultaneous blowdowns change Figure \ref{fig:linksimplify7} to Figure \ref{fig:linksimplify8} via a process akin to that shown in Figure \ref{fig:prehexstep2}.  Finally, Figure \ref{fig:linksimplify8} can be isotoped to Figure \ref{fig:prehexsquare5}, at which point we rejoin the previous argument. 

\begin{figure}
     \labellist
\small\hair 2pt
\pinlabel  $n$ at 220 45
\pinlabel  $-n$ at 225 290
\pinlabel  $[1]$ at 405 210
\pinlabel  $[-1]$ at 400 100
%\pinlabel  $0$ at 180 255
%\pinlabel  $0$ at 25 155
\endlabellist
 \centering
    \includegraphics[scale=0.5]{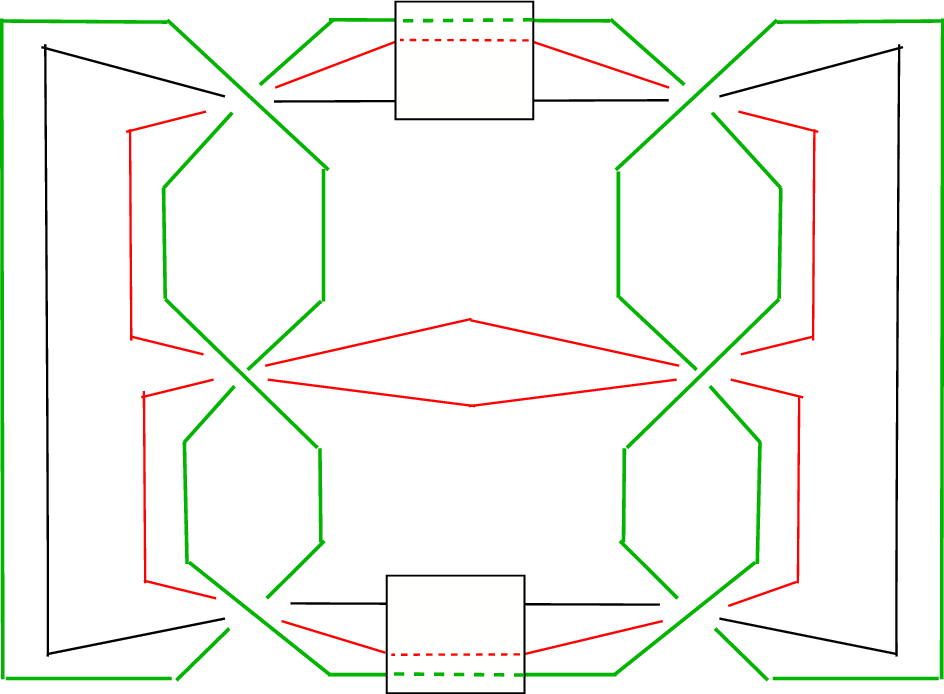}
    \caption{} \label{fig:linksimplify7}
    \end{figure}

\begin{figure}
 \centering
    \includegraphics[scale=0.5]{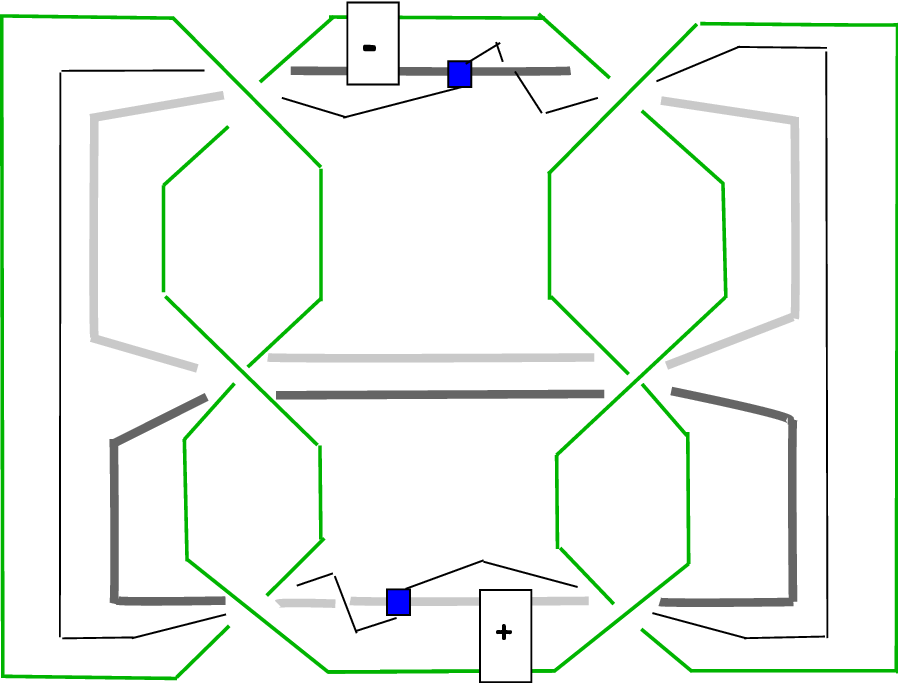}
    \caption{} \label{fig:linksimplify8}
    \end{figure}

\section{Identifying the torus $T \subset M$} \label{sect:torus}

It is unsatisfying that in both views of the construction above it is hard to identify the torus $T \subset M$, whose critical properties are listed in the introductory section.  $T$ appears in \cite[Figure 19]{GST} (here in Figure \ref{fig:GSTFig19}), but the appearance is well before the bracketed red curves are blown down, and it is hard to track $T$ through that operation.  The most obvious torus in $M$ is the swallow-follow $T_{sf}$ torus in $S^3 - Q$, which is also the mapping torus of the green circle in Figure \ref{fig:Mmon}.  This torus does indeed intersect $V_0$ in two points, but a little experimentation shows that Dehn twisting $V_0$ along curves in this torus produces links much simpler than the $V_n$ (in fact mostly links obtained from $Q \cup V_0$ just by twisting $Q$, as described in the remarks before Theorem \ref{thm:class}).  

Here is a way to see a more complicated candidate for the torus $T$:  Just as the trefoil knot contains a spanning M\"obius band, the manifold $M$ contains an interesting Klein bottle, the mapping cylinder of the six brown arcs in Figure \ref{fig:tor1a}.  Other tori in $M$ can be obtained from $T_{sf}$ by Dehn twisting it along the Klein bottle.  (This must be done in the direction of the  curve in the Klein bottle whose complement is orientable, in order for the operation to makes sense).  An example is shown in Figure \ref{fig:tor1b}:  the union of the twelve green arcs is a circle that is preserved, with its orientation, by the monodromy, so its mapping cylinder is a torus $T$ in $M$. We now verify that $T$ is the torus we seek.

\begin{figure}
\centering
\subfigure[]%caption in brackets
{
  \includegraphics[scale=0.41]{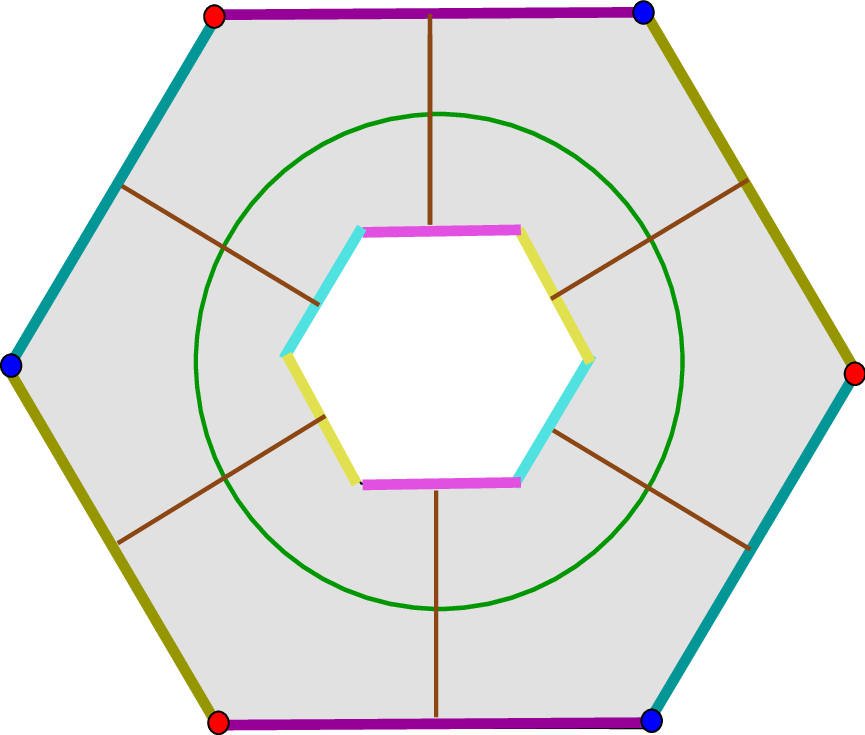} \label{fig:tor1a}
}
\subfigure[]%caption in brackets
{
 \includegraphics[scale=0.41]{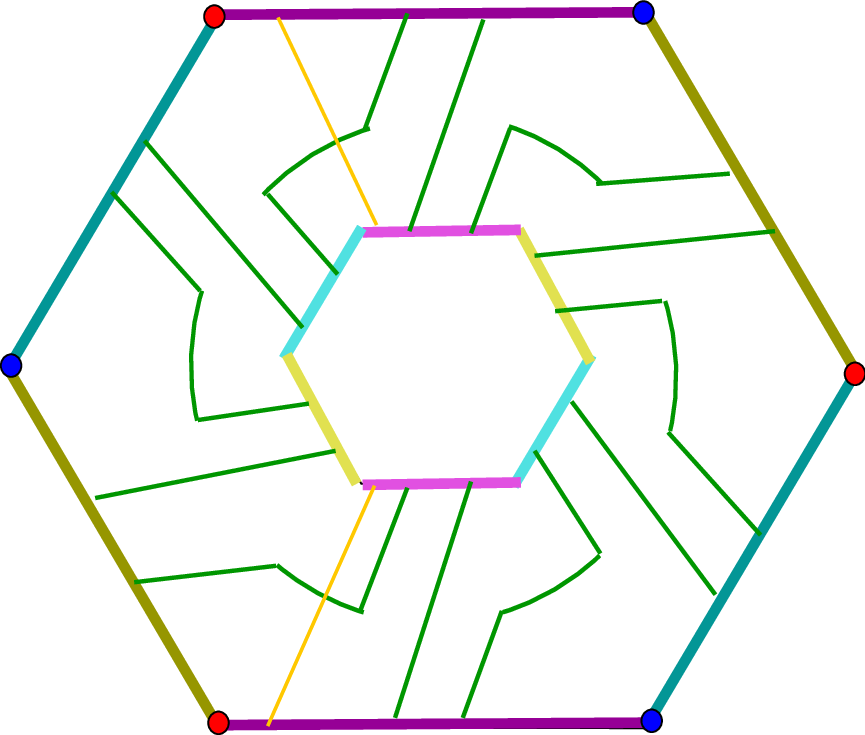} \label{fig:tor1b}
}
\caption{}
\label{fig:tor1ab}
\end{figure}

\begin{figure}
 \centering
    \includegraphics[scale=0.5]{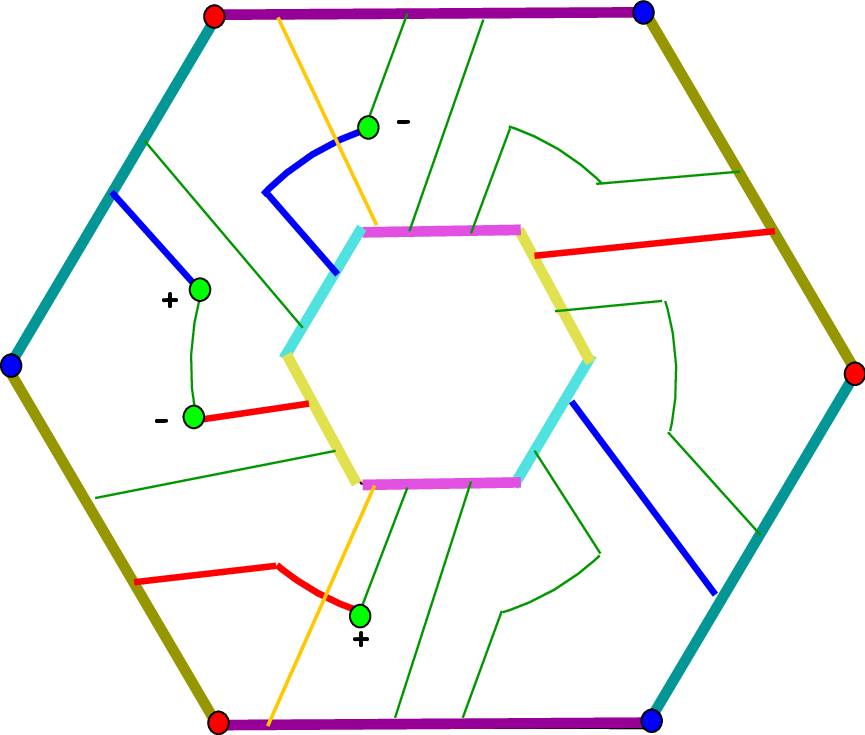}
    \caption{} \label{fig:tor1c}
    \end{figure}

The yellow curve in Figure \ref{fig:tor1b} is the original component $V_0$; it intersects $T$ twice.  Figure \ref{fig:tor1c} shows (in red 
and blue) two parallel simple closed curves on $T$:  As was the convention in Section \ref{sect:hex} (see also Figure \ref{fig:xsection}) imagine both the blue  arc and the red arc lying in $F_\cup \times \{ \frac12 \} \subset F_\cup \times I \subset M$.  A green dot at the end of each arc is labeled $\pm$ and represents in each case a vertical arc that ascends (resp. descends) to $F_\cup \times \{ 1 \}$ (resp $F_\cup \times \{ 0 \}$).  The two ends of the blue arc (and similarly the two ends of the red arc) are then identified in $M$ by the monodromy.  The 
framing of these curves given by the normal direction to $T$ is clearly the ``blackboard framing" given by the figure, so Dehn twisting 
along $T$ in the direction of these curves can be visualized by widening the blue and red arcs into bands, and then Dehn twisting along 
the bands.  Finally, visibly identify the ends of the blue band (and the ends of the red band) by altering the monodromy near the 
central circle, as was done in Figure \ref{fig:Qmon}, to get red and blue annuli.  The result is Figure \ref{fig:tor2a}, which uses train-track merging to show the direction of the Dehn twisting.  The result is essentially identical to that shown in Figure \ref{fig:hexp1} and can be made 
identical by flipping both red and blue annuli along their core curves (see Figure \ref{fig:tor2b}).  Why the colored annuli have 
framing $\pm 1$ in $S^3$ (that is, before $0$-surgery is done on $Q$ to create $M$) is explained briefly  in 
Section \ref{sect:hex} via Figure \ref{fig:bookopen}.

\begin{figure}
\centering
\subfigure[]%caption in brackets
{
  \includegraphics[scale=0.4]{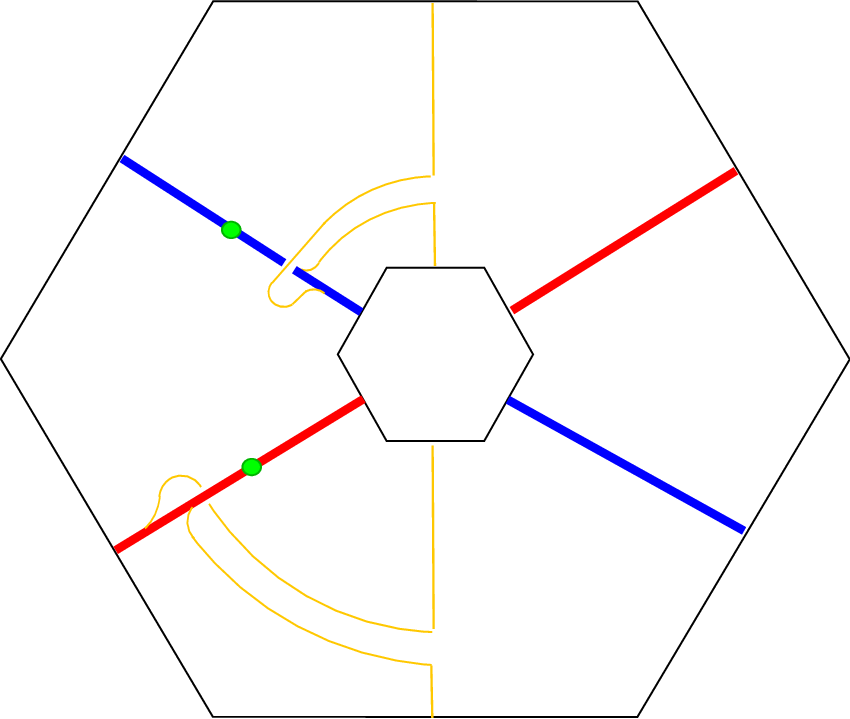} \label{fig:tor2a}
}
\subfigure[]%caption in brackets
{
 \includegraphics[scale=0.4]{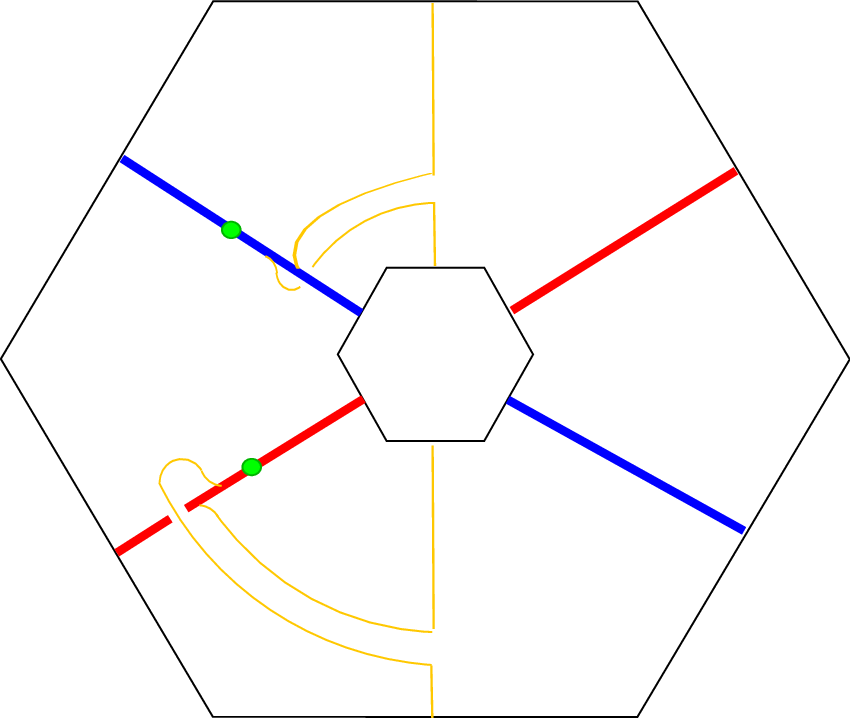} \label{fig:tor2b}
}
\caption{}
\label{fig:tor2ab}
\end{figure}

This shows that constructing the links $L_n$ is simple.  The fact that all these 
$L_n$ satisfy Weak Generalized Property R requires only the argument in \cite[Section 10]{GST}.  In the next section we show directly that a natural presentation of the trivial group  given by the handle-structure of the $4$-manifold cobordism between $S^3$ and $\#_2(S^1 \times S^2)$ is
$$<a, b\;|\;aba = bab, a^n = b^{n+1}>$$
and so $L_n$ is unlikely to satisfy Generalized Property R, on Andrews-Curtis grounds.   This direct calculation makes it possible to sidestep the full complexity of the 
original construction in \cite{GST}, with its roots in \cite{Go}.  

\section{How the cobordism presents the trivial group}\label{sect:present}

The $4$-manifold cobordism $W$ between $S^3$ and $\#_2(S^1 \times S^2)$ given by surgery on $L_n = Q \cup V_n$ is decomposed by the $3$-manifold $M$ into two pieces:  
\begin{itemize}
\item the cobordism between $S^3$ and $M$ obtained by $0$-framed surgery on the square knot $Q$ and
\item the cobordism between $M$ and $\#_2(S^1 \times S^2)$ given by further surgery on the curve $V_n$, lying in the fiber as described above, with framing given by the fiber.
\end{itemize}
Let $W_L$ denote the cobordism between $\bdd_0 W_L = \#_2(S^1 \times S^2)$ and $\bdd_1 W_L = M$. and  $W_Q$ denote the cobordism between $\bdd_0 W_Q = M$ and $\bdd_1 W_Q = S^3$.

\subsection{A natural presentation of $\pi_1(W_L)$.}

The cobordism $W_L$ consists of a single $2$-handle attached to a collar of $\#_2(S^1 \times S^2)$.  A presentation for its fundamental group is naturally obtained in two steps:
\begin{enumerate}
\item Choose a non-separating pair of normally oriented non-parallel $2$-spheres $S_a, S_b \subset \bdd_0 W_L$.  
\item  Write down the word $r$ in letters $a, \overline{a}, b, \overline{b}$ determined by the order and the orientation with which the attaching circle for the $2$-handle intersects the spheres $S_a, S_b$.
\end{enumerate}
Then the presentation is simply $$\pi_1(W_L) = <a, b \;| \; r>.$$
Different choices of $2$-spheres will give rise to different relators $r$, hence different presentations.  But once these $2$-spheres are chosen, the word $r$ is determined (up to conjugation) by the free homotopy class of the attaching circle for the $2$-handle.

\begin{lemma} The construction above gives a natural choice of $2$-spheres $S_a, S_b \subset \bdd_0 W_L$ so that the associated presentation is
$$\pi_1(W_L) = <a, b\;|\;aba = bab>.$$
\end{lemma}

\begin{proof}  Somewhat counterintuitively, the key to finding the relevant $2$-spheres is to consider the properties of the curve $V_n$ lying in a fiber $F_\cup \subset M$, for $W_L$ can be viewed (dually) as the cobordism obtained by attaching a $2$-handle to $M$ along $V_n$.   Put another way: 
observe that $\bdd_0 W_L = \#_2(S^1 \times S^2)$ is obtained from $M$ by replacing a tubular neighborhood $N$ of $V_n$ by a solid torus $N'$ whose meridian circle is parallel in $\bdd N$ to the curve  $\bdd N \cap F_\cup$, where $F_\cup$ is the fiber of $M$ on which $V_n$ lies.  

Recall the structure of $M$ from Section \ref{sect:hex}:  it is the mapping torus of a monodromy $\rho: F_\cup \to F_\cup$ which can be viewed as $\pi/3$ rotation on a hexagon in which opposite sides have been identified to create the genus two surface $F_\cup$.  See Figures \ref{fig:Mmon} and \ref{fig:rhosigma}.  An important feature is that the three curves $V_n, \rho(V_n), \rho^{-1}(V_n)$ are all disjoint in $F_\cup$ and together the three curves divide $F_\cup$ up into two pairs of pants, $P_a$ and $P_b$.  It is these pairs of pants,  capped off by meridian disks of $N'$, that will constitute the spheres $S_a, S_b \subset \bdd_0 W_L$.  

To be a bit more concrete, let $A \subset F_\cup$ be the annulus neighborhood $N \cap F_\cup$ of the curve $V_n$ in the fiber $ F_\cup$.  Then $\rho(A)$ and $\rho^{-1}(A)$ are disjoint from each other and from $A$.  Denote their complementary components in $F_\cup$ by $P_a$ and $P_b$, each a pair of pants.  Once $N$ is replaced by $N'$ to get $\bdd_0 W_L$, each component of $\bdd A$ bounds a meridian disk in $N'$.  On the other hand, the mapping cylinder structure $$M \cong F_\cup \times I/(x,1) \sim (\rho(x), 0)$$ shows that each component of $\bdd \rho(A)$ or $\bdd \rho^{-1}(A)$ is also parallel (via a vertical annulus) to a meridian circle in $\bdd N'$.  These vertical annuli describe how each boundary component of $P_a$ (resp. $P_b$) can be capped off by a disk in $\bdd_0 W_L$ to create spheres $S_a$ and $S_b$.  

\begin{figure}
 \centering
    \includegraphics[scale=0.8]{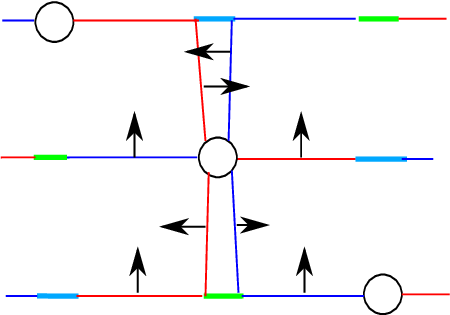}
    \caption{} \label{fig:sasb}
 \end{figure}

To illustrate, Figure \ref{fig:sasb} shows a schematic picture of the infinite cyclic cover of $M$.  The three horizontal lines represent three lifts of the fiber $F_\cup$, which we might think of as $\rho(F_\cup), F_\cup$ and $\rho^{-1}(F_\cup)$.  The circle in each fiber represents a lift of $V_n$; the short horizontal green and aqua segments represent lifts of spanning arcs of the annuli $\rho(A)$ and $\rho^{-1}(A)$.  The red lines represent lifts of $P_a$ and the blue lines represent lifts of $P_b$.  The focus is on the central circle, the order in which copies of $P_a$ and  $P_b$ occur around that circle, and how $P_a$ and $P_b$ are normally oriented by a choice of normal orientation on $F_\cup$.  The circle represents a meridian circle of $\bdd N$, so the order and orientation of intersections with $P_a \subset S_a$ and $P_b \subset S_b$ gives the relation $r$ in the presentation we seek.  If we start by reading just below the circle and proceed clockwise the relation is $$aba\overline{bab} \implies aba = bab$$ as predicted.  Keeping in mind that $P_a, P_b$ are connected surfaces in $M$ itself, it's easy to see how to find a simple closed curve that is disjoint from $P_b$ (resp $P_a$) and intersects $P_a$ (resp $P_b$) in a single point. (For example, in Figure \ref{fig:sasb}, the upper-right (northeast) quadrant of the meridian circle of $\bdd N$ is an arc whose endpoints lie near and on the same side of $P_a$ and which intersects $P_b$ in a single point.) It follows that $S_a$ and $S_b$ are non-parallel and that their union is non-separating, as we require.  

In the argument above it may appear that the schematic diagram (at least) depends on taking $n = 0$.  It is certainly easier to understand the picture for the special case in which the surgery curve is $V_0$.  But notice that there is a natural homeomorphism of $M$ to itself which takes $V_0$ to $V_n$, namely twisting $n$ times about the torus $T$ defined in Section \ref{sect:torus}, so the presentation given by the handle structure on $W_L$ is the same, regardless of $n$.  
\end{proof}

\subsection{The relator from $W_Q$.}

When $W_Q$ is attached to $W_L$ along $M$ a relator is added.  It is determined by a meridian circle $\mu$ of the square knot as it appears in $\bdd_1 W_Q = S^3$.  In $M$ $\mu$ is represented by an arc in $F_\cup \times I$ that connects $point \times \{ 1 \}$ to $\rho(point) \times \{ 0 \}.$  In the schematic Figure \ref{fig:sasb} the lift of $\mu$ is covered by the tilted arc shown in Figure \ref{fig:Qrelator}.  For the illustrated case $n = 0$, the circle is seen to represent the relator $b\overline{a}a$ or simply $b$.  

\begin{figure}
 \centering
    \includegraphics[scale=0.8]{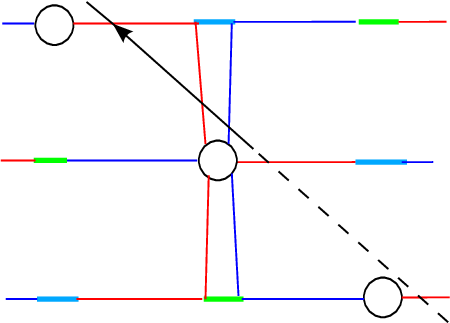}
    \caption{} \label{fig:Qrelator}
 \end{figure}
 
The general case, in which $V_0$ becomes $V_n$ by twisting around a torus $T$ (see Section \ref{sect:torus}) it is quite complicated to calculate how the twisting alters the way in which $P_a$ and $ P_b$ (and so $S_a$ and $ S_b)$ intersect $\mu$.  The reader is invited to try the calculation directly in, say, Figure \ref{fig:deg9ex}. Rather than trying to understand how twisting around $T$ affects $P_a$ and $P_b$, so that we might calculate the resulting intersections of $\mu$ with $S_a$ and $S_b$, we will instead leave  $V_0, P_a, P_b$ unchanged, but twist the curve $\mu$ around $T$ and determine how the twisted $\mu$ intersects the unchanged $P_a$ and $P_b$.  

The first observation is that the schematic presentation of $P_a, P_b$ given in Figures \ref{fig:sasb} and \ref{fig:Qrelator} is in fact an accurate depiction of how the swallow-follow torus $T_{sf}$ in $M$ (i. e. the mapping torus of the circle depicted in Figure \ref{fig:tor1a}) intersects the curve $V_0$ and the planar surfaces $ P_a$ and $P_b$.  To be specific, Figure \ref{fig:PcapT} shows how $V_0$, $ P_a$ and $P_b$ intersect the lift $\tilde{T}_{sf} $ of $T_{sf}$ to the infinite cyclic cover of $M$ .  (Identify the right side of the figure with the left to get $\tilde{T}_{sf} \cong S^1 \times \R$.)
The torus $T$ is obtained by twisting $T_{sf}$ along the Klein bottle $K$, as described in Section \ref{sect:torus}.  Since $K$ can be made disjoint from all three curves $V_0, \rho(V_0)$ and $\rho^{-1}(V_0)$,  Figure \ref{fig:PcapT} likewise depicts the pattern of intersection of  the lift $\tilde{T}$  with $V_0$, $P_a$ and $P_b$. 

\begin{figure}
 \centering
    \includegraphics[scale=0.8]{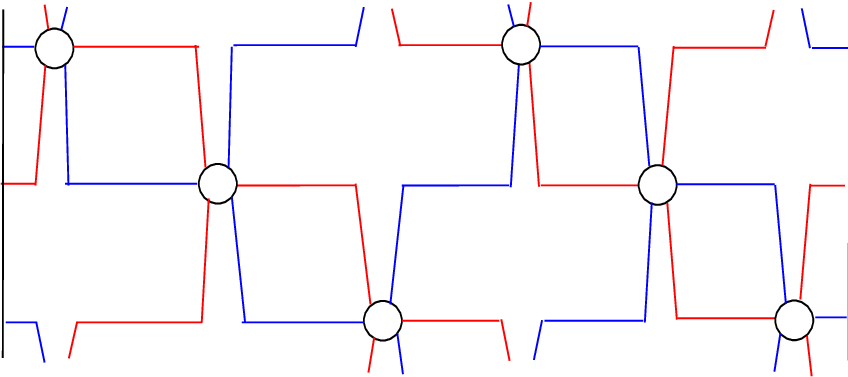}
    \caption{} \label{fig:PcapT}
 \end{figure}
 
 Now consider two adjacent vertical faces in the $hexagon \times I$ from which $M$ was constructed in Section \ref{sect:hex}.  Pick, say, the faces whose tops appear as the two left-most edges of the hexagon in Figure \ref{fig:Mmon}.  Figure \ref{fig:walls} depicts these two faces; the monodromy identifies the top of the left face with the bottom of the right face.  The gray vertical bars represent the intersection of a bicollar of the torus $T$ with these faces; the bars have opposite normal orientation (see Figure \ref{fig:tor1b}).  The purple diagonal represents the meridian $\mu$.  We are trying to see how $\mu$ intersects $P_a$ and $P_b$ when $\mu$ is twisted along $T$ as $\mu$ passes through the bicollar.  The slope at which $\mu$ is twisted along $T$ is derived easily from Figure \ref{fig:tor1c}: in words, a single twist, corresponding to $n = 1$ will move a  point vertically from the bottom of $F \times I$ to the top of $F \times I$ as it moves the point horizontally $\frac16$ of the way around $T$. 

\begin{figure}
 \centering
    \includegraphics[scale=0.8]{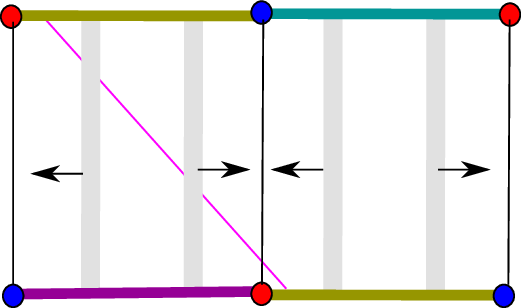}
    \caption{} \label{fig:walls}
 \end{figure}
 
Suppressing a bit of detail, Figure \ref{fig:twisting} shows the resulting trajectory of $\mu$ during its two arcs of passage through the bicollar of the torus (i. e. the two gray bands in Figure \ref{fig:walls}).  The case depicted is $n = 2$; higher values of $n$ would give longer purple arcs in Figure \ref{fig:twisting}.  The two purple arcs are oriented in opposite directions in the figure since $T$ has opposite normal orientation at the two points at which $\mu$ passes through the bicollar.  The word that results from twisting $\mu$ as it passes through the bicollar can now be read off:  $bb^{n}\overline{a}^n = 1$ or $a^n = b^{n+1}$. The first $b$ comes from the intersection of $\mu$ with the bottom of $F \times I$ in Figure \ref{fig:walls}.

\begin{figure}
 \centering
    \includegraphics[scale=0.8]{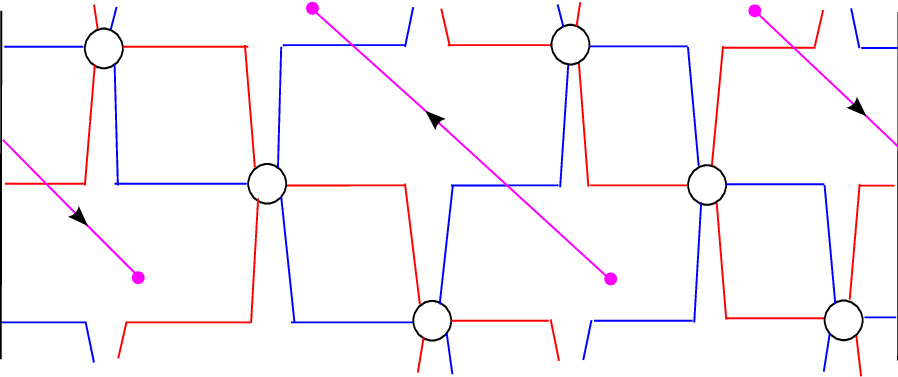}
    \caption{} \label{fig:twisting}
 \end{figure}

The final presentation for the fundamental group of the cobordism is then
$$\pi_1(W) = <a, b\;|\;aba = bab, a^n = b^{n+1}>,$$ as expected.  

\section{Understanding $T \cap S^3$}\label{sect:TinS3}

Recall that $M$ is obtained from $S^3 - \eta(Q)$ by attaching a solid torus to $\bdd \eta(Q)$ with framing zero.  The torus $T \subset M$ described in Section \ref{sect:torus} passes through the added solid torus in two meridians.  This can 
be seen from Figure \ref{fig:tor1b}: An arc connecting a generic point in $F_\cup$ to its image under the monodromy (i. e. $\pi/3$ rotation of the hexagon) will intersect $T$ twice; and the union of such an arc with the mapping torus of the generic 
point represents the core of the attached solid torus  To put it another way, $T \cap (S^3 - \eta(Q))$ is a twice-punctured torus whose boundary circles are longitudes of $Q$.  Visualizing $T$ in $S^3 - \eta(Q)$ is not easy.  

We begin with trying to 
understand how the Klein bottle $K$ used in the construction of $T$ intersects $S^3 - \eta(Q)$.  In outline, this is easy to see:  the three colored circles in Figure \ref{fig:squaremon} are cyclically permuted by the monodromy and the mapping 
torus of this action (whose cube is orientation reversing) defines the Klein bottle.  But the argument above, applied to $K$, shows that $K$ passes once through the added solid torus; this outline masks how this happens, since it ignores the effect 
shown in Figure \ref{fig:Qmon}.  The actual monodromy on $F \subset (S^3 - \eta(Q))$ and how $K$ is determined by this monodromy can be understood by starting with Figure \ref{fig:posthexsquare1} which, when doubled by reflection along its right side, gives all of $F$.  The Klein bottle $K$ intersects $F$ in Figure \ref{fig:posthexsquare1} in the red, blue and green arcs.  

If Figure \ref{fig:posthexsquare1} were completed to a trefoil knot by adding a line on the right, as shown in red on the top left of Figure \ref{fig:KleinS31}, the effect of the monodromy would be relatively easy to see as a screwing motion along a vertical axis.  The trajectory of a typical point is shown by the dotted arrow; the effect of the monodromy on $F$ is shown on the top right of Figure \ref{fig:KleinS31}.  But notice that the red line in the trefoil knot has been moved; it can be moved back to its previous location, so the picture can be doubled to give $F$, by a twist on a collar of the trefoil knot.  This final monodromy is shown on the lower left of Figure \ref{fig:KleinS31}.  Recall how in Figure \ref{fig:posthexsquare1}, $K$ was isotoped off the capping off disk, so that $K \cap F$ consists of three circles instead of six arcs. The same thing can be accomplished after the mondromy; the result is shown on the bottom right of Figure \ref{fig:KleinS31}.  Notice that in $F$ the monodromy has moved the red and green circles to where respectively the green and blue circles were in Figure \ref{fig:posthexsquare1}.  The blue circle is now much more complicated,

\begin{figure}
 \centering
    \includegraphics[scale=0.4]{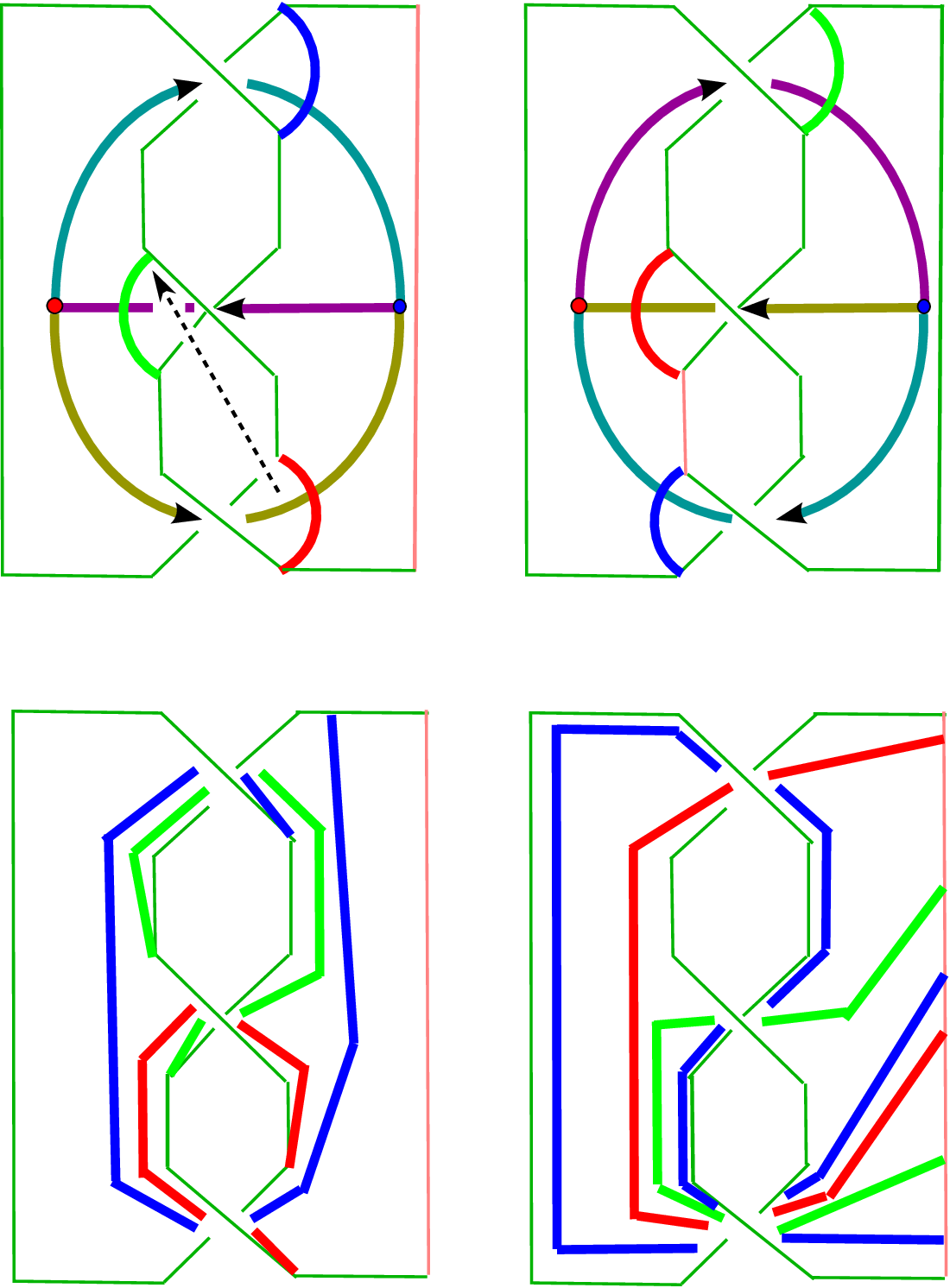}
    \caption{} \label{fig:KleinS31}
 \end{figure}
 
  Recall how in Figure \ref{fig:posthexsquare1}, $K$ was isotoped off the capping off disk, so that $K \cap F$ consists of three circles instead of six arcs. The same thing can be accomplished after the mondromy; the result is shown on the bottom right of Figure \ref{fig:KleinS31}.  Notice that in $F$ the monodromy has moved the red and green circles to where respectively the green and blue circles were in Figure \ref{fig:posthexsquare1}.  The blue circle in $F$ is now much more complicated, but sliding it once over a meridian of the solid torus that is attached to $Q$ with framing $0$, as illustrated in Figure \ref{fig:KleinS32}, isotopes the blue circle to the same position as the red circle before the mondromy.  

\begin{figure}
 \centering
    \includegraphics[scale=0.35]{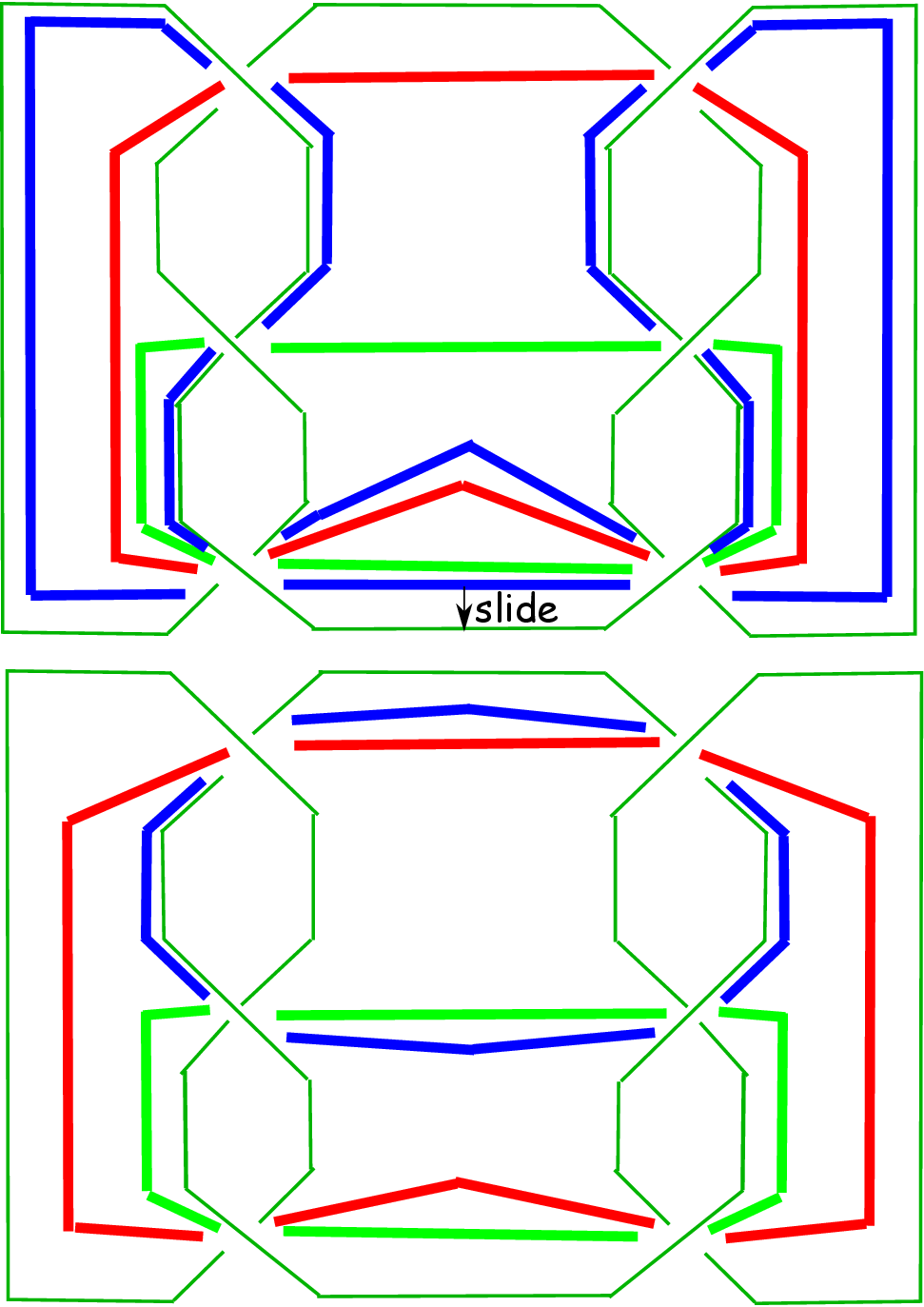}
    \caption{} \label{fig:KleinS32}
 \end{figure}
 
 It will be useful to have tracked how the monodromy moves a point near the Klein bottle.  This is illustrated in Figure \ref{fig:KleinS33}.  The aquamarine dot is tracked through all the moves involved in understanding the monodromy, moving first backwards from Figure \ref{fig:posthexsquare2} to Figure \ref{fig:posthexsquare1}, then through Figure \ref{fig:KleinS31}.  The resulting trajectory can be isotoped rel end points to the simpler path shown in the last panel of Figure \ref{fig:KleinS33}.  Note that it can be completed to a meridian of the knot by adding an arc in $F$ from the blue dot to the aquamarine dot, puncturing $K$ exactly once.  This is what we expect from the mapping cylinder view of $M$ shown in Figure \ref{fig:tor1a}.  
 
 \begin{figure}
 \centering
    \includegraphics[scale=0.6]{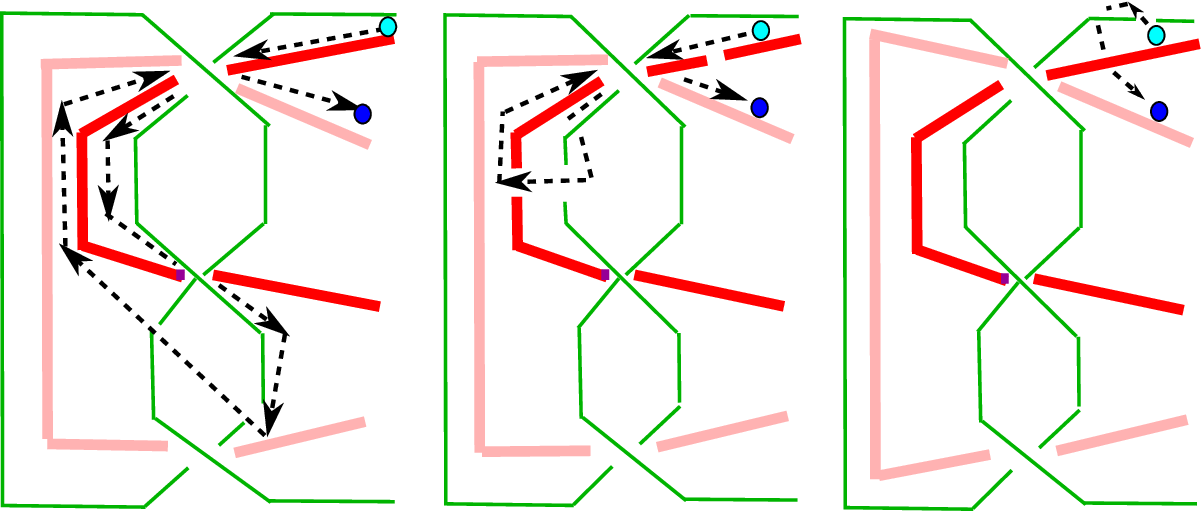}
    \caption{} \label{fig:KleinS33}
 \end{figure}
 
 The torus $T$ we seek is obtained by doing a Dehn twist of the swallow-follow torus $T_{sf}$  along $K$.  This is difficult to picture, but at least some of it can be seen directly from the figure, in particular the way in which $T$ intersects the Seifert surface $F$.  $T_{sf}$ can be thought of as the union of 
 \begin{itemize}
 \item a plane perpendicular to the plane of the figure, one that intersects  $F$ in a vertical line and is punctured twice by the knot; 
 \item and an annulus neighborhood of that half of the knot that lies to one side of the vertical plane (depicted as the right side of the vertical arc in Figure \ref{fig:TorS3a1}). 
 \end{itemize}
 Then $T$ is obtained by Dehn-twisting $T_{sf}$ along $K$.  The resulting curve $T \cap F$ is shown in Figure \ref{fig:TorS3a1}.

It is fairly straightforward now to identify within $T$ the slope along which we Dehn twist to change $V_n$ to $V_{n+1}$, as described in the introductory section.   $T$ is the mapping torus of the monodromy on $T \cap F$, and a track of that monodromy away from $F$ is shown in Figure \ref{fig:KleinS33}.  If that track is combined with an arc in $F$ connecting its endpoints, the result is a curve lying in $T$.  Applying this construction to two tracks that lie on opposite sides of $K$ in Figure \ref{fig:KleinS33} gives us two disjoint curves in $T$; these are shown in Figure \ref{fig:TorS3b1}.  In other words, the dotted circles in Figure \ref{fig:TorS3b1} are two of the the simple closed curves in $T$ that were earlier identified in Figure \ref{fig:tor1c}.  The pair of circles can then be manipulated as shown in Figure \ref{fig:TorS3b4} until they are positioned as the gray annuli we recognize from Figure \ref{fig:prehexsquare0}.  (The pair of arrows in each of the top two panels of Figure \ref{fig:TorS3b4} are meant to clarify the isotopy that moves the pair of curves to the position in the next panel.)  
  
  \begin{figure}
 \centering
    \includegraphics[scale=0.5]{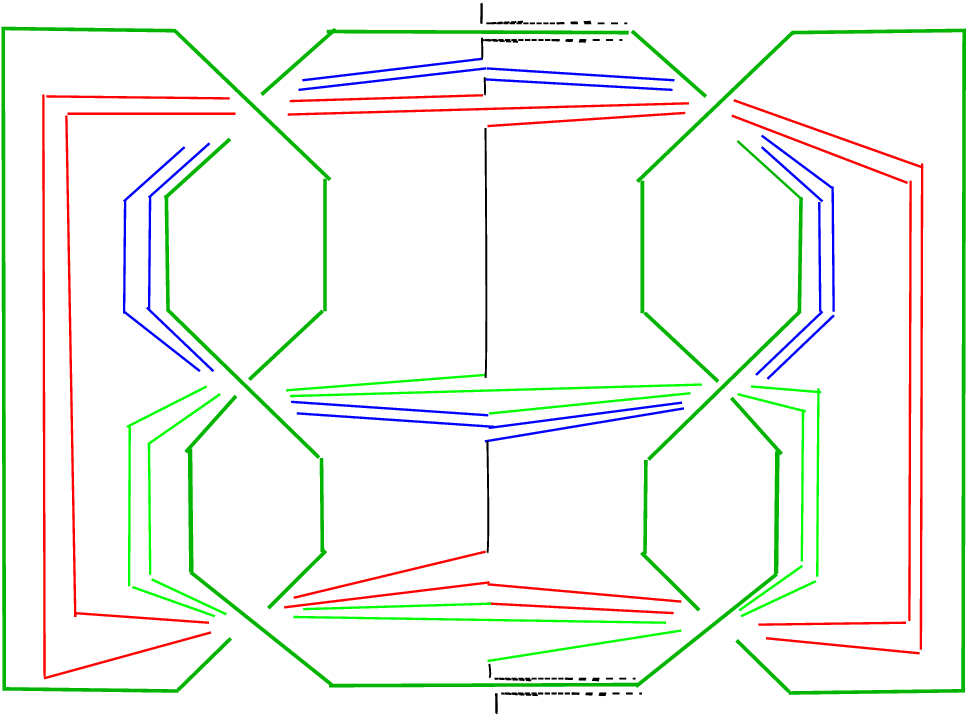}
    \caption{} \label{fig:TorS3a1}
 \end{figure}
 
   \begin{figure}
 \centering
    \includegraphics[scale=0.5]{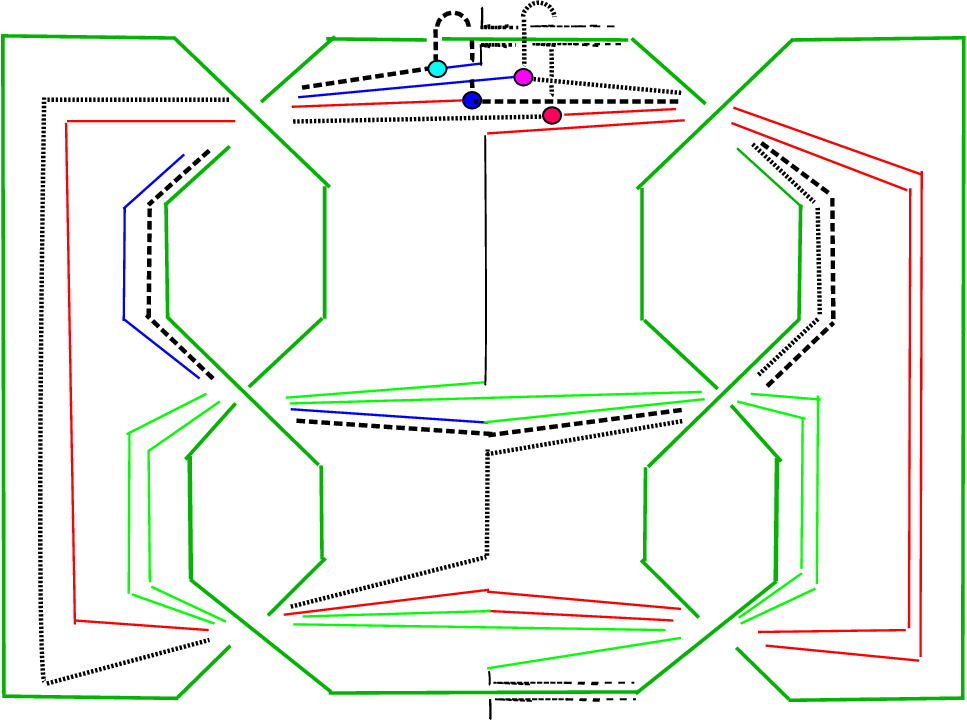}
    \caption{} \label{fig:TorS3b1}
 \end{figure}
 
    \begin{figure}
 \centering
    \includegraphics[scale=0.7]{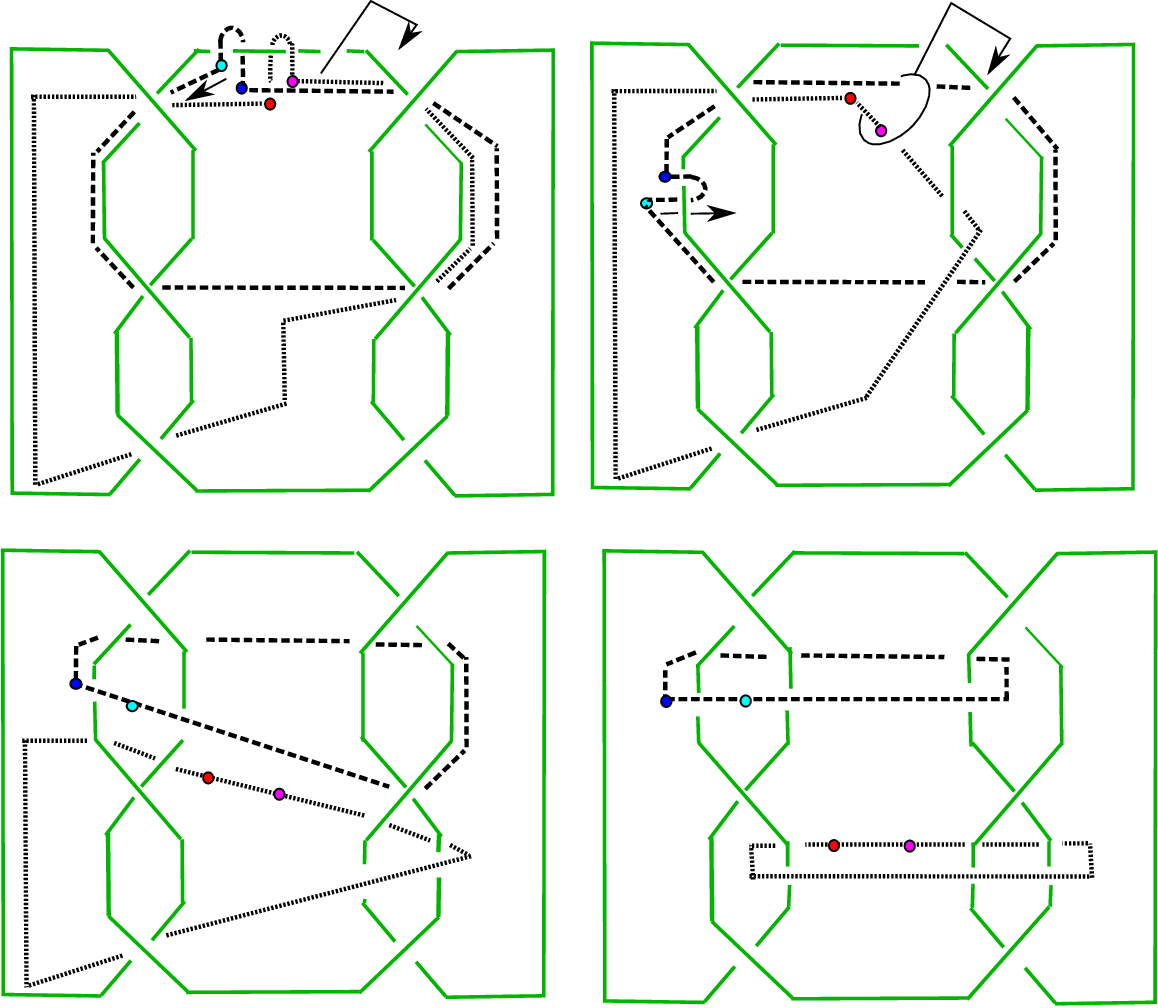}
    \caption{} \label{fig:TorS3b4}
 \end{figure}

%The same monodromy applied to a point in $T$ near (in $F$) to the previous point gives rise to another simple closed curve in $T$ parallel to the one described above, see Figure \ref{fig:TorS32}.  Figure \ref{fig:TorS34} shows how to isotope it in $S^3 - \eta(Q)$ to the other circle in Figure \ref{fig:prehexsquare0}. 
%
%  
%  \begin{figure}
% \centering
%    \includegraphics[scale=0.5]{TorS32}
%    \caption{} \label{fig:TorS32}
% \end{figure}
% 
%   \begin{figure}
% \centering
%    \includegraphics[scale=0.7]{TorS34}
%    \caption{} \label{fig:TorS34}
% \end{figure}

\end{document}